\documentclass[11pt,reqno]{amsart}
\title[Fusion-stable structures]
{Fusion-stable structures on triangulated categories}
\author{Yu Qiu}
\address{Qy:
	Yau Mathematical Sciences Center and Department of Mathematical Sciences,
	Tsinghua University,
    100084 Beijing,
    China.
    \&
    Beijing Institute of Mathematical Sciences and Applications, Yanqi Lake, Beijing, China}
\email{yu.qiu@bath.edu}

\author{Xiaoting Zhang}
\address{Zx:
    Beijing Advanced Innovation Center for Imaging Theory and Technology, Academy for Multidisciplinary Studies, Capital Normal University, Beijing 100048, China}
\email{xiaoting.zhang09@hotmail.com}
\usepackage[a4paper]{geometry}
\usepackage{float}
\usepackage[usenames,dvipsnames]{xcolor}
\usepackage{subfigure}

\usepackage[colorlinks, linkcolor=blue,anchorcolor=Periwinkle,
   citecolor=red,urlcolor=Emerald]{hyperref}
\usepackage{amsmath}           
\usepackage{amssymb}           
\usepackage{mathabx}
\usepackage{latexsym}          
\usepackage{mathrsfs}          
\usepackage{bbm}
\usepackage{bookmark}
\usepackage{enumitem}
\usepackage{array}
\usepackage{graphicx}
\usepackage{cleveref}

\usepackage{tikz}
\usetikzlibrary{matrix}
\usepackage{url}

\usetikzlibrary{decorations.markings,cd}
\usetikzlibrary{arrows}
\usetikzlibrary{decorations.pathreplacing,decorations.pathmorphing,shadings,fadings,calc}
\usepackage{extarrows}

\tikzset{->-/.style={decoration={  markings,  mark=at position #1 with
    {\arrow{>}}},postaction={decorate}}}
\tikzset{-<-/.style={decoration={  markings,  mark=at position #1 with
    {\arrow{<}}},postaction={decorate}}}
\tikzcdset{arrow style=tikz, diagrams={>=stealth}}

\theoremstyle{plain}
\newtheorem{theorem}{Theorem}[section]

\newtheorem{lemma}[theorem]{Lemma}

\theoremstyle{definition}
\newtheorem{definition}[theorem]{Definition}

\newtheorem{example}[theorem]{Example}

\newtheorem{remark}[theorem]{Remark}

\numberwithin{equation}{section}

\def\hua{\mathcal}
\def\hh{\mathcal}

\def\<{\langle}
\def\>{\rangle}
\def\={\simeq}

\def\to{\rightarrow}

\def\NN{\mathbb{N}}
\def\ZZ{\mathbb{Z}}

\def\CC{\mathbb{C}}

\def\RR{\mathbb{R}}

\renewcommand{\k}{\mathbf{k}}
\def\C{\hh{C}}
\newcommand{\D}{\operatorname{\hh{D}}}

\def\dim{\operatorname{dim}}

\def\Add{\operatorname{add}}
\def\Aut{\operatorname{Aut}}

\def\Sim{\operatorname{Sim}}
\def\Hom{\operatorname{Hom}}

\def\End{\operatorname{\mathcal{E}nd}}
\def\FP{\operatorname{FPdim}}
\def\Ext{\operatorname{Ext}}

\def\Stab{\operatorname{Stab}}

\def\Stap{\operatorname{Stab}^\circ}
\def\diff{\operatorname{d}}
\def\Br{\operatorname{Br}}

\def\rep{\operatorname{rep}}

\def\deg{\operatorname{deg}}
\newcommand{\h}{\hh{H}}            

\newcommand{\EG}{\operatorname{EG}}       
\newcommand{\EGp}{\operatorname{EG}^\circ}       
\newcommand{\SEG}{\operatorname{SEG}}       
\newcommand{\SEGp}{\operatorname{SEG}^\circ}       
\newcommand{\CEG}{\operatorname{CEG}}             
\newcommand{\uCEG}{\underline{\CEG}}

\renewcommand{\mod}{\operatorname{mod}}

\newcommand{\Cone}{\operatorname{Cone}}

\newcommand{\id}{\operatorname{id}}

\newcommand{\per}{\operatorname{per}}
\newcommand{\pvd}{\operatorname{pvd}}

\newcommand{\ST}{\operatorname{ST}}        

\def\ww{node[white]{$\bullet$}node{$\circ$}}
\newcommand{\on}[1]{\operatorname{#1}}

\def\hX{\hua{X}}
\def\hY{\hua{Y}}

\def\hT{\hua{T}}
\def\hF{\hua{F}}
\def\hP{\hua{P}}
\def\hQ{\hua{Q}}

\def\Fu{\mathcal{G}} 
\def\KG{\on{\mathsf{K}}}
\def\TLJ{\mathsf{TLJ}}
\def\wt{\mathbf{w}}         
\def\vv{\mathbf{v}}         
\def\spec{\mathbf{Q}}        
\def\wtF{\mathsf{fold}_\wt}
\def\unfold{\mathsf{unfold}}
\def\GNQ{\Gamma_N Q}
\def\DQ{\D_\infty(Q)}
\def\VG{\on{\mathsf{vec}}^G(\k)}
\def\wtV{\mathbf{V}}
\def\wtW{\mathbf{W}}
\def\wtf{\mathbf{f}}

\def\bfi{\mathbf{i}}
\def\bfj{\mathbf{j}}
\def\bfk{\mathbf{k}}
\def\bfl{\mathbf{l}}
\def\bfa{\mathbf{a}}
\def\bfp{\mathbf{p}}
\def\bfn{\mathbf{n}}
\def\bfs{\mathbf{s}}
\def\WT{\Xi}
\def\1{\mathbf{1}}
\def\2{\mathbf{2}}
\def\3{\mathbf{3}}
\def\4{\mathbf{4}}

\def\CYGQ{\Fu_N(\overline{\spec})}

\newcommand{\tlj}[2]{[#2]_{#1}}

\begin{document}

\def\Qy#1{\textcolor{blue}{#1}}
\setlength\parindent{0em}
\setlength{\parskip}{5pt}

\begin{abstract}
Let $\mathcal{G}$ be a fusion category acting on a triangulated category $\mathcal{D}$, in the sense that $\mathcal{D}$ is a $\mathcal{G}$-module category. Our motivation example is fusion-weighted species, which is essentially Heng's construction. We study $\mathcal{G}$-stable tilting, cluster and stability structures on $\mathcal{D}$. In particular, we prove the deformation theorem for $\mathcal{G}$-stable stability conditions.

A first application is that Duffield-Tumarkin's categorification of cluster exchange graphs of finite Coxeter-Dynkin type can be naturally realized as fusion-stable cluster exchange graphs. Another application is that the universal cover of the hyperplane arrangements of any finite Coxeter-Dynkin type can be realized as the space of fusion-stable stability conditions for certain ADE Dynkin quiver.
This provides an alternative uniform proof of $K(\pi,1)$-conjecture in the finite Coxeter-Dynkin case.
\end{abstract}

\maketitle

\tableofcontents\addtocontents{toc}{\setcounter{tocdepth}{1}}

\section{Introductions}

In representation theory of algebras, tilting theory plays important roles.
There is a couple of closely related versions, e.g.
t-structures/hearts with Happen-Reiten-Smal{\o} tilting and
$m$-cluster tilting/silting with (forward) mutations.
They are used in the study of triangle/derived equivalences
from Morita theory, homological mirror symmetry,
additive categorification of Fomin-Zelevinsky's cluster algebras and more (see e.g. \cite{Ke2,Toen,AI12,KQ1}).

The stability structure on triangulated (and abelian) categories is another class of key structures.
Originally coming from geometric invariant theory,
Bridgeland introduced the triangulated version and showed that all stability conditions on a triangulated category $\D$ form a complex manifold $\Stab(\D)$.
The motivation also comes from the study of $\Pi$-stability of D-branes from string theory
with various applications, e.g. Donaldson-Thomas theory, cluster theory and Teichm\"{u}ller theory (cf. \cite{B1,B2,KQ2,QW}).

When there is a finite group $G$ acting on an abelian/triangulated category $\C$, 
the $G$-equivariant (abelian/triangulated) category is well-defined, cf. \cite{Ela},
but the $G$-invariant/$G$-stable ones (as orbit categories) may not, cf. \cite{Ke0,CQ}.
In some special cases (e.g. for acyclic quivers),
this can be done via folding technique (and Frobenius morphisms/functors, cf. \cite{DD}).
The tilting and stability structures on the folded ($G$-stable) category $\C/G$
can be naturally identified with the $G$-stable tilting and stability structures on $\C$, cf. \cite{CQ}.
Then many statements can be extended to more general cases, e.g.
from simply laced ADE Dynkin case to non-simply laced BCFG cases, cf. \Cref{thm:KQ}.

On the other hand, fusion categories were introduced by Etingof-Nikshych-Ostrik \cite{ENO},
as a natural generalization of finite groups and their behavior over $\CC$.
It plays an important role in the study of Hopf algebras and quantum groups \cite{EGNO}.
In the work of Heng \cite{H1},
fusion actions on abelian categories (e.g. representations of a quiver) are introduced and studied
to prove a generalization of Gabriel theorem for all finite Coxeter-Dynkin (A to I) type quivers.
Here, a fusion $\Fu$-action on a category $\C$ means that $\C$ is a $\Fu$-module category.
A closely related work is Duffield-Tumarkin's categorifications of non-integer quivers \cite{DT1,DT2},
where they use weighted-folding technique to categorify cluster structure of non-simply laced type quivers (in particular, type H and I).

Inspired by \cite{H1,H2,DT1,DT2},
we systematically study various fusion-stable tilting/stability structures on triangulated categories.
Besides generalizing many known results with ease, we also prove the fusion-stable deformation theorem on Bridgeland stability conditions (\Cref{thm:stab}).
As applications, we prove the following two results.

\begin{figure}[ht]\centering
\def\snn{node{\small{$\bullet$}}}
\def\sww{node[white]{\small{$\bullet$}}node{\small{$\circ$}}}
\makebox[\textwidth][c]{
\begin{tikzpicture}[scale=.6]
\begin{scope}[rotate=90]
\draw(-90:.2)coordinate (v);
\foreach \j in {0,...,2}{
\draw[thick]($(120*\j:4)+(v)$) coordinate (a\j)
            ($(120*\j:2.3)-(v)$) coordinate (c\j)
            (120*\j+40:4) coordinate (b\j)
            (120*\j-40:4) coordinate (d\j) ;
\draw[dashed,gray,thick,fill=orange!99, opacity=0.1]
    (a\j)to(b\j)to(c\j)to(d\j)to(a\j);
\draw[thick](a\j)to(b\j) (d\j)to(a\j);}
\draw[fill=cyan, opacity=0.1] (d0)to(c0)to(b0)to(d1)to(c1)to(b1)to(d2)to(c2)to(b2)to(d0);
\draw[fill=cyan!50, opacity=0.1] (d0)to(a0)to(b0)to(d1)to(a1)to(b1)to(d2)to(a2)to(b2)to(d0);
\draw[dashed, gray] (d0)to(c0)to(b0)to(d1)to(c1)to(b1)to(d2)to(c2)to(b2)to(d0);

\draw[thick](v)edge(a0)edge(a1)edge(a2)\snn
    (b0)to(d1)(b1)to(d2)(b2)to(d0)
    (a0)\snn(a1)\snn(a2)\snn  (b0)\snn(b1)\snn(b2)\snn (d0)\snn(d1)\snn(d2)\snn;
\draw[dashed,gray]($(0,0)-(v)$)edge(c1)edge(c2)edge(c0)\sww
     (c0)\sww(c1)\sww(c2)\sww;
\end{scope}
\begin{scope}[shift={(9,0)}]
\foreach \j in {1,...,12}
    {\draw[thick](30*\j:4) coordinate (t\j) to (30*\j+30:4) \snn;}
\foreach \j in {1,...,4}
    {\draw[](90*\j:2.9) coordinate (s\j) \snn;}
\draw[](135:1.5) coordinate (x1) (-45:1.5) coordinate (x2);
\draw[gray](-135:2.3) coordinate (y1) (45:2.3) coordinate (y2);
\draw[fill=cyan, opacity=0.2] (t3)to(y2)to(t12)to(t1)to(t2)to(t3);
\draw[fill=cyan, opacity=0.2] (t9)to(y1)to(t6)to(t7)to(t8)to(t9);
\draw[fill=cyan, opacity=0.1] (x1)to(s1)to(t4)to(t5)to(s2)to(x1);
\draw[fill=cyan, opacity=0.1] (x2)to(s3)to(t10)to(t11)to(s4)to(x2);

\draw[fill=orange!99, opacity=0.1] (t2)to(s1)to(t4)to(t3)to(t2);
\draw[fill=orange!99, opacity=0.1] (t5)to(s2)to(t7)to(t6)to(t5);
\draw[fill=orange!99, opacity=0.1] (t8)to(s3)to(t10)to(t9)to(t8);
\draw[fill=orange!99, opacity=0.1] (t11)to(s4)to(t1)to(t12)to(t11);

\draw[dashed,gray](t2)to(s1)to(t4) (t5)to(s2)to(t7)
     (t8)to(s3)to(t10) (t11)to(s4)to(t1)
     (s1)to(x1)to(s2) (s3)to(x2)to(s4) (x1)to(x2);
\draw[thick]
     (t3)to(y2)to(t12) (t6)to(y1)to(t9) (y1)\snn to(y2)\snn
     (x1)\sww(x2)\sww(s1)\sww(s2)\sww(s3)\sww(s4)\sww;
\end{scope}
\begin{scope}[shift={(18,0)}]
\draw(-90:.2)coordinate (v);
\foreach \j in {0,...,18}
    {\draw[thick](20*\j:4.1) coordinate (t\j);}
\foreach \j in {0,3,...,18}
    {\draw[gray,dashed,fill=orange!99, opacity=0.1]
        (20*\j+20:4.1)to (20*\j:4.2) coordinate (t\j)to(20*\j-20:4.1)to(20*\j:3.2) coordinate (s\j);
     \draw[gray,dashed]
        (s\j)to(20*\j+20:4.1)to(t\j)to(20*\j-20:4.1)to(s\j) \sww;}

\draw[gray](-90:.2)coordinate (x)\sww;
\foreach \j in {0,...,2}
    {\draw[dashed,gray](x)to($(30+120*\j:2)$)coordinate(v\j)\sww;}

\draw[gray,dashed,fill=cyan, opacity=0.1](s0)to(v0)to(s3)to(t2)to(t1)to(s0);
\draw[gray,dashed,fill=cyan, opacity=0.1](s6)to(v1)to(s9)to(t8)to(t7)to(s6);
\draw[gray,dashed,fill=cyan, opacity=0.1](s12)to(v2)to(s15)to(t14)to(t13)to(s12);
\draw[thick](90:.2)coordinate (y);
\foreach \j in {0,...,2}
    {\draw[thick](y)to($(90+120*\j:3.2)$)coordinate(w\j);}

\draw[gray,dashed](s0)to(v0)to(s3)to(t2)to(t1)to(s0)
                (s6)to(v1)to(s9)to(t8)to(t7)to(s6)
                (s12)to(v2)to(s15)to(t14)to(t13)to(s12);

\draw[thick,fill=cyan, opacity=0.25](w0)to(t3)to(t4)to(t5)to(t6);
\draw[thick,fill=cyan, opacity=0.25](w1)to(t9)to(t10)to(t11)to(t12);
\draw[thick,fill=cyan, opacity=0.25](w2)to(t15)to(t16)to(t17)to(t0);

\draw[thick](t3)to(w0)to(t6) (t9)to(w1)to(t12) (t15)to(w2)to(t18)
    (t0)to(t1)to(t2)to(t3)to(t4)to(t5)to(t6)to(t7)to(t8)to(t9)to
    (t10)to(t11)to(t12)to(t13)to(t14)to(t15)to(t16)to(t17)to(t18)
    (y)\snn(w0)\snn(w1)\snn(w2)\snn;
\draw[gray](v0)\sww(v1)\sww(v2)\sww(x)\sww;
\foreach \j in {0,...,18}{\draw[thick](t\j)\snn;}
\foreach \j in {0,3,...,18}{\draw[gray](s\j)\sww;}
\end{scope}
\end{tikzpicture}
}
\caption{The cluster exchange graphs $\uCEG(\Delta)$ for type $A_3/B_3/H_3$.\\
Type $B_3/H_3$ can be obtained by weighted folding type $D_4/D_6$, cf. \Cref{fig:BCEI} and \Cref{fig:H234}, respectively.
}\label{fig:ABH}
\end{figure}

\begin{itemize}
  \item[\textbf{Setup}]
    Let $\Delta$ be any finite type Coxeter-Dynkin (type A to I) quiver and
    $\Omega$ be a simply laced finite type quiver with respect to some fusion $\Fu$-action
    (not necessarily unique or connected, cf. \Cref{def:unfolding})
    with a weighted folding $\wtF\colon Q \to \spec$.
\end{itemize}
\begin{description}
  \item[\Cref{thm:ceg}] The cluster exchange graph of type $\Delta$
  can be realized as the $\Fu$--stable cluster exchange graph of type $\Omega$, cf. \Cref{fig:ABH}
and \Cref{ex:w-folding}.

  \item[\Cref{thm:Kleinian}]
   The universal cover of the fundamental domain of hyperplane arrangements of type $\Delta$
   can be realized as the space of $\Fu$-stable stability conditions of type $\Omega$,
   which is a finite type component and hence contractible.
\end{description}
Note that \Cref{thm:Kleinian} gives an alternative uniform proof of the $K(\pi,1)$-conjecture,
which was originally conjectured by Deligne in \cite{D}, cf. see \cite{Par} for more details, and
cf. more recent works \cite{QW,AW} (that reproves ADE case and some non-simply laced cases).

Another remark is that the weighted folding has been studied in various aspects:
\begin{itemize}
  \item for root systems (cf. \Cref{fig:D6=2H3} and \Cref{fig:E8=2H4}, which is well-known),
  \item for braid groups (cf. \Cref{thm:Cr} of \cite{Cr}) and
  \item for cluster combinatorics/generalized associahedrons (cf. \cite{FZ,FR}).
\end{itemize}
Basically we need to put the first two of these together with the fusion construction consistently to obtain the second application above.

\subsection*{Acknowledgement}
We would like to thank Alastair King for inspiring discussions
during our visit to Bath in July-August of 2023 while most of this work was done.
Qy also thanks Bernhard Keller and Han Zhe for sharing their expertise during collaborating related works.
This paper owes a debt to the work of Heng.

This work is supported by
National Natural Science Foundation of China (Grant No.12425104, No. 12101422) and
National Key R\&D Program of China (No.2020YFA0713000).

\section{Preliminaries}
\subsection{Basics}
\paragraph{\textbf{Additive categories}}\

We fix an algebraically closed field $\k$ and all categories are assumed to be $\k$-linear.

In an additive category $\C$ with a subcategory $\hua{B}$,
we denote by
\[
    \hua{B}^{\perp_{\C}}\colon= \left\{C\in\C \mid \Hom(B,C)=0, \forall B\in \hua{B}\right\}.
\]
We may write $\hua{B}^{\perp}$ when $\C$ is implied. The subcategory ${}^{\perp_{\C}}\hua{B}$ is defined similarly.
For subcategories $\hT,\hF$ of an abelian category $\h$,
denote by
$$
    \hT*\hF\colon= \{M\in\h\mid\text{$\exists$ s.e.s.
        $0\to T\to M\to F\to0$ s.t. $T\in\hT, F\in\hF$} \}.
$$
If additionally $\Hom(\hT,\hF)=0$, then we write $\hT\perp\hF$ for $\hT*\hF$.
For a subcategory $\C$ of $\h$,
denote by $\<\C\>$ the full subcategory of $\h$ consisting of objects that admit a filtration of s.e.s. with factors in $\C$.

A morphism $f\colon E\to F$ in an arbitrary additive category is called \emph{right minimal} if
it does not have a direct summand of the form $E'\to0$.
Similarly for \emph{left minimal}.

Let $\C$ be an additive category with a full subcategory $\hX$.
A \emph{right $\hX$-approximation} of an object $C$ in $\C$ is a morphism $f\colon X\to C$ with $X\in\hX$,
such that
\[\Hom(X',f)=f\circ{}_-\colon\on{Hom}(X',X)\to\on{Hom}(X',C)\]
is an epimorphism for any $X'\in \hX$.
Dually, a \emph{left $\hX$-approximation} of an object $C$ in $\C$ is a morphism $g\colon C\to X$ with $X\in\hX$,
such that
\[\Hom(g,X')={}_-\circ{}g\colon\on{Hom}(X,X')\to\on{Hom}(C,X')\]
is an epimorphism for any $X'\in\hX$.
A \emph{minimal right/left $\hX$-approximation} is a right/left $\hX$-approximation that is also right/left minimal.
We say $\hX$ is \emph{contravariantly finite} (resp. \emph{covariantly finite}) if every object in $\C$ admits a right (resp. left) $\hX$-approximation.
It is \emph{functorially finite} if it is both contravariantly finite and covariantly finite.

\paragraph{\textbf{Coxeter diagram}}\

A finite Coxeter diagram $\Delta$ is one of the (edge-)weighted graphs listed in \Cref{fig:Cox},
where the weights are assumed to be 3 if not written.
In general, a Coxeter diagram is an edge-valued graph $\Delta$ with function $\vv\colon\Delta_1\to\ZZ_\ge3$.

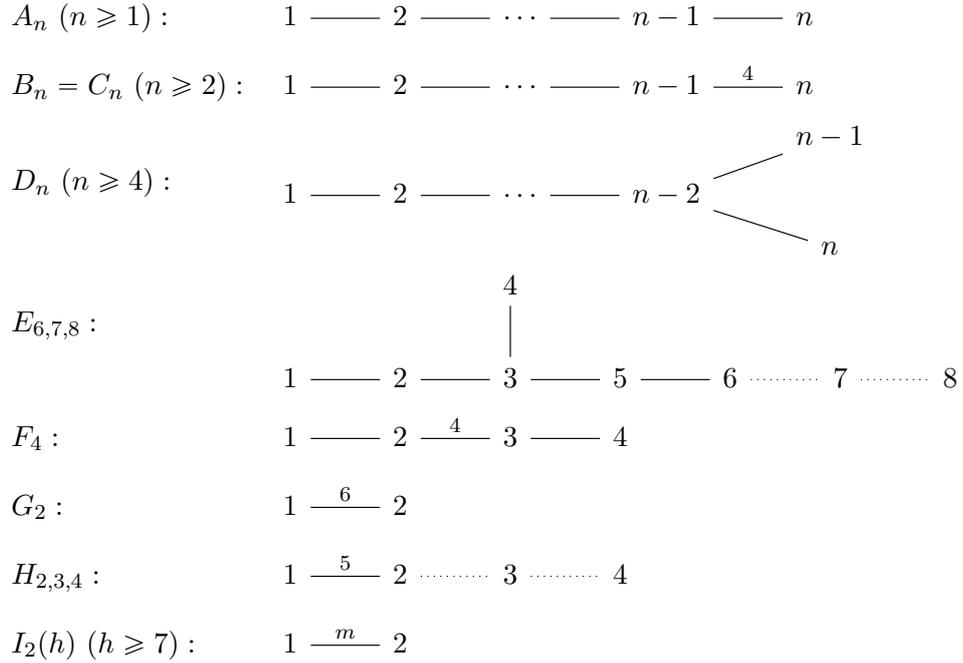
\begin{figure}[htpb]\centering
\[\renewcommand{\arraystretch}{2}
		\begin{array}{llr}
			A_{n}\ (n\geq 1): \quad &
            \begin{tikzcd}[every arrow/.append style={no head}]
              1 \ar[r]& 2 \ar[r]& \cdots \ar[r]&n-1\ar[r]& n
            \end{tikzcd}\\
			B_{n}=C_n\ (n\geq 2): \quad &
            \begin{tikzcd}[every arrow/.append style={no head}]
              1 \ar[r]& 2 \ar[r]& \cdots \ar[r]&n-1\ar[r,"4"]& n
            \end{tikzcd}\\
			D_{n}\ (n\geq 4): \quad &
            \begin{tikzcd}[every arrow/.append style={no head},row sep=.5pc]
                &&&& n-1\\
              1\ar[r] & 2 \ar[r]& \cdots \ar[r]&n-2\ar[ur]\ar[dr]\\
                &&&& n
            \end{tikzcd}\\
			E_{6,7,8}: \quad &
            \begin{tikzcd}[every arrow/.append style={no head}]
                &&4\ar[d]\\
              1 \ar[r]& 2 \ar[r]& 3 \ar[r] & 5 \ar[r] & 6 \ar[r,dotted] & 7 \ar[r,dotted] & 8
            \end{tikzcd}\\
			F_4: \quad &
            \begin{tikzcd}[every arrow/.append style={no head}]
              1 \ar[r]& 2 \ar[r,"4"]& 3 \ar[r] & 4
            \end{tikzcd}\\
			G_2: \quad &
            \begin{tikzcd}[every arrow/.append style={no head}]
              1 \ar[r,"6"]& 2
            \end{tikzcd}\\
            H_{2,3,4}: \quad &
            \begin{tikzcd}[every arrow/.append style={no head}]
              1 \ar[r,"5"]& 2 \ar[r,dotted]& 3 \ar[r,dotted] & 4
            \end{tikzcd}\\			I_2(h)\ (h\geq 7): \quad &
            \begin{tikzcd}[every arrow/.append style={no head}]
              1 \ar[r,"h"]& 2
            \end{tikzcd}
		\end{array}
\]
\caption{A complete list of finite Coxeter graphs}
\label{fig:Cox}
\end{figure}

\paragraph{\textbf{Braid and Weyl groups}}\

\begin{definition}\label{def:W.Br}
The (Artin) braid group $\Br_\Delta$ associated to $\Delta$ is defined by the presentation
\begin{gather}\label{eq:BrD}
    \Br_\Delta\colon= \< b_{\bfi}\mid \bfi\in\Delta_0 \>\Big/
        (\Br^{m_{\bfi,\bfj}}(b_{\bfi},b_{\bfj})\mid \forall \bfi,\bfj\in\Delta_0),
\end{gather}
where
\[
  m_{\bfi,\bfj}=
    \begin{cases}
        2, & \mbox{if there is no edge beteween $\bfi$ and $\bfj$};\\
        \vv(\alpha), & \mbox{if there is exactly one edge $\bfi \frac{\alpha}{\quad} \bfj$}; \\
        \infty,& \mbox{if there are multi edges between $\bfi$ and $\bfj$}.
    \end{cases}
\]
$\Br^{m}(a,b)$ is the relation $\underbrace{aba\cdots}_{m}=\underbrace{bab\cdots}_{m}$
and by convention, $\Br^{\infty}(a,b)$ means no relation.

The Weyl group $W_\Delta$ is defined by the presentation
\begin{gather}\label{eq:WD}
    W_\Delta\colon= \< b_{\bfi}\mid \bfi\in\Delta_0 \>\Big/
        ( b_{\bfi}^2=1, \Br^{m_{\bfi,\bfj}}(b_{\bfi},b_{\bfj})\mid \forall \bfi,\bfj\in\Delta),
\end{gather}
\end{definition}

There is a canonical surjection $\pi_\Delta\colon\Br_\Delta\to W_\Delta$ sending the standard generators to the standard ones.

For a general acyclic quiver $Q$, one can also define the associated braid group $\Br_Q$,
for its underlygraph with trivial function $\vv\equiv3$.
\subsection{Various structures}

\paragraph{\textbf{Torsion pairs, t-structures and hearts}}\

A \emph{torsion pair} in an abelian category $\h$ consists of two additive subcategories $\hT$ and $\hF$ such that $\h=\hT\perp\hF$.
Denote by $\Sim\h$ the set of simples of $\h$. An abelian category $\h$ is called a \emph{length category} if any object in $\h$ admits a finite filtration with factors in $\Sim\h$. The abelian category $\h$ is called $\emph{finite}$ if it is a length category with a finite number of simples.

Similarly, we have the corresponding notions and notations in a triangulated category $\D$, replacing s.e.s. by triangles.
A \emph{t-structure} $\hP$ of $\D$ is the torsion part of some torsion pair of $\D$ which is closed under shift (i.e.~$\hP[1]\subset\hP$).
We will write it as $\D=\hP\perp\hP^{\perp}$.
Such a t-structure is \emph{bounded} if for any object $M$ in $\D$, we have $M[\gg0]\in\hP$ and $M[\ll0]\in\hP^\perp$.

The \emph{heart} of a t-structure $\hP$ is $\h=\hP^\perp[1]\cap\hP$.
When $\hP$ is bounded, its heart is a heart of the triangulated category $\D$, in the sense that
\begin{itemize}
    \item $\Hom(\h[a],\h[b])=0$ for any $a>b$ in $\ZZ$;
    \item any object $E$ in $\D$ admits an HN-filtration of triangles,
    whose factors are $H_i[k_i]$ for some $H_i\in\h$, $1\le i\le l$, and integers $k_1>\cdots>k_l$.
\end{itemize}
Note that any heart of a triangulated category is abelian.

A heart is \emph{finite} if it is a length category and with finitely many simples.
\paragraph{\textbf{Silting/cluster tilting}}\

A \emph{silting subcategory} $\hY$ in a triangulated category $\D$ is
an additive subcategory such that $\Ext^{>0}(\hY,\hY)=0$ and $\D=\on{thick}\hY$.
Here $\on{thick}\hY$ is the smallest triangulated subcategory of $\D$ that is closed under taking direct summands and contains $\hY$.

For an integer $m\ge2$, an $m$-cluster tilting subcategory ($m$-CTS) in $\D$ is
an additive subcategory such that $\Ext^{i}(\hY,\hY)=0$, for $1\le i\le m$,
and $\D=\hY[1]*\cdots*\hY[m]$.

We will unify these two notions, i.e. using $\infty$-cluster to mean silting.
In the case that $\hY=\Add\mathbf{Y}$ for a basic object $\mathbf{Y}$,
we call $\mathbf{Y}$ an \emph{$m$-cluster tilting object}.
Here, $\Add\mathbf{Y}$ is the smallest additive subcategory that is closed under taking direct summands and contains $\mathbf{Y}$.
\paragraph{\textbf{Tilting and mutation}}\

We proceed to recall a couple of tilting structures.
The first one is (simple) Happel-Reiten-Smal{\o} tilting of t-structures.

\begin{definition}\cite[Def.~3.7]{KQ1}
Let $\h=\hT\perp\hF$ be a torsion pair in a heart $\h$ of a triangulated category $\D$.
Then the \emph{forward} (resp. \emph{backward}) \emph{tilt} $\h^\sharp$ (resp. $\h^\flat$)
with respect to $\hT\perp\hF$ is the heart that admits a torsion pair
$\h^\sharp=\hF[1]\perp\hT$ (resp. $\h^\flat=\hF\perp\hT[-1]$).
Denote the operation by
\[
    \mu^\sharp_\hF\colon\h\to\h^\sharp\quad\text{and}\quad\mu^\flat_\hT\colon\h\to\h^\flat.
\]

The forward/backward tilting is \emph{simple} if $\hF$/$\hT$ equals $\Add S$
for a simple object $S\in\Sim\h$.
Denote by $\mu^{\sharp/\flat}_S$ the simple tilting in such a case.

The exchange graph $\EG(\D)$ has hearts as vertices and simple tiltings as (directed) edges.
\end{definition}

The second tilting structure is (ind) mutation of $m$-CTS.
\begin{definition}\cite[Def.~2.30]{AI12}
Let $\hY$ be an $m$-CTS of $\C$.
Let $\hX$ be an (additive) covariantly finite subcategory of $\hY$.
The forward mutation of $\hY$ with respect to $\hY\setminus\hX$ is
\[
    \mu_{\hY\setminus\hX}^\sharp(\hY)\colon=
        \Add\left\{\hX\cup\{ \mu_{\hX}^\sharp(Z) \}_{Z\in\hY\setminus\hX} \right\},
\]
where $\mu_{\hY\setminus\hX}^\sharp(Z)=\Cone(f)$ is defined by a minimal left $\hX$-approximation $f\colon Z \to \widehat{X}$, for some $\widehat{X}\in\hX$.
Here, $\hY\setminus\hX$ is the additive subcategory of $\hY$ consisting of objects without summands in $\hX$.

We say such a mutation is indecomposable (or ind for short) if $\hY\setminus\hX=\Add Z$ for some indecomposable object $Z$ and one writes $\mu_{Z}^\sharp=\mu_{\hY\setminus\hX}^\sharp$.

The $m$-cluster exchange graph $\CEG_m(\C)$ of $\C$ has $m$-cluster tilting subcategories as vertices and ind mutations as (directed) edges.
As before, silting means $\infty$-cluster tilting and $\SEG(\C)=\CEG_\infty(\C)$.
\end{definition}

\paragraph{\textbf{Stability structure}}\

Let $\mathbb{H}=\{ \RR_{>0}e^{\bfi \pi \phi}\mid \phi\in (0,1] \}\subset\CC$ be the upper half plane
and $\KG(\D)$ denote the Grothendieck group of a triangulated category $\D$.

\begin{definition}\cite[Def.~1.1]{B1}
A stability condition $\sigma=(Z,\hP_\RR)$ on a triangulated category $\D$ consists of
a central charge $Z\in\Hom_\ZZ( \KG(\D),\CC )$ and a slicing, i.e. a $\RR$-collection $\hP_\RR=\{\hP(\phi)\mid \phi\in\RR\}$ of additive (abelian in fact) subcategories $\hP(\phi)$, satisfying:
\begin{itemize}
  \item $Z$ is compatible with $\hP_\RR$, in the sense that, for any $E\in\hP(\phi)$,
    $Z(E)=m(E)\cdot e^{\bfi \pi \phi}$ for some $m(E)\in\RR_{>0}$.
  \item $\hP_\RR$ is compatible with shifts, i.e. $\hP(\phi+1)=\hP(\phi)[1].$
  \item $\Hom(\hP(\phi_1),\hP(\phi_2))=0$ if $\phi_1>\phi_2$.
  \item Any object $M$ in $\D$ admits a Harder-Narasimhan (HN) filtration, that is a chain of triangles:
  \begin{equation}\label{eq:HN}\begin{tikzcd}[column sep=.8pc]
    0=M_0 \ar[rr] && M_1 \ar[rr]\ar[dl] && M_2 \ar[r]\ar[dl] & \ldots \ar[r] & M_{l-1} \ar[rr] && M_l=M \ar[dl] \\
        &A_1 \ar[ul,dashed]&&A_2 \ar[ul,dashed]&&&& A_l \ar[ul,dashed]
  \end{tikzcd}\end{equation}
  with $A_i\in\hP(\phi_i)$ such that $\phi_1>\phi_2>\cdots>\phi_l$.
\end{itemize}
together with the technical property, known as the support property, cf. \cite[\S~3.1]{IQ}.
\end{definition}
For object $M$ with HN-filtration \eqref{eq:HN}, it has
an upper phase $\phi^+(M)=\phi_1$ and a lower phase $\phi^-(M)=\phi_l$.
Objects in $\hP(\phi)$ are called $\sigma$-semistable or semistable (with respect to $\sigma$) with phase
$\phi(E)=\phi^{\pm}(E)=\phi$ and mass $m(E)$.
Simple objects in $\hP(\phi)$ are called $\sigma$-stable.

The heart $\h_\sigma$ of a stability condition $\sigma=(Z,\hP_\RR)$ is $\hP(0,1]$
for $\hP(I)=\<\hP(\phi)\mid\phi\in I\>$, where $I$ is any interval.
Equivalently, a stability condition $\sigma$ is given by a heart $\h$ and
a stability function $Z$ on $\h$, i.e. a group homomorphism $Z\colon\KG(\h)\to\CC$,
satisfying compatibility (i.e. $Z(H)$ is in $\mathbb{H}$ for $H\in\h$) and HN-property.
As $\KG(\D)\cong\KG(\h)$ for any heart $\h$ of $\D$, the stability function coincide with the central charge.

The key theorem of Bridgeland \cite{B1} shows that all stability conditions on a triangulated category $\D$ form a complex manifold $\Stab(\D)$ with local coordinates given by the central charge $Z$.

Given a finite heart $\h$,
all stability conditions $\sigma$ supported on $\h$ (i.e. stability conditions $\sigma=(Z,\hP_\RR)$ satisfying $\hua{P}(0,1]=\h$) form a complex $n$-cell $U(\h)\cong\mathbb{H}^{\Sim\h}$, cf. \cite[\S~2.2]{B1}.

A \emph{finite-type component} $\Stab_0$ in $\Stab(\D)$ is of the form
\begin{equation}\label{eq:0}
    \Stab_0=\bigsqcup_{\h\in\EG_0} U(\h),
\end{equation}
where $\EG_0$ is a connected component of the exchange graph of hearts in $\D$,
such that any heart in it is finite and has finitely many torsion pairs.
See \cite{QW} for examples and more details.

\subsection{The quiver categories}\label{sec:quivery}\

Given a graded quiver $Q$ (i.e. there is a map from the set $Q_1$ of arrows to $\ZZ$) and let $\GNQ$ be the associated (degree N) Ginzburg dg algebra,
constructed as follows (\cite{Ke1}):
\begin{itemize}
\item Let $\overline{Q}^N$ be the $N$-Calabi-Yau double of $Q$,
that is, the graded quiver whose vertex set is $Q_0$ and whose arrows are: the arrows in $Q_1$ (with inherited degrees);
an arrow $a^*\colon j\to i$ in degree $2-N-\deg(a)$ for each arrow $a\colon i\to j$ in $Q_1$ and
a loop $e^*\colon i\to i$ in degree $1-N$ for each vertex $e$ in $Q_0$.
\item The underlying graded algebra of $\GNQ$ is the completion of the graded path algebra $\k\overline{Q}^N$
with respect to the ideal generated by the arrows of $\overline{Q}^N$.
\item  The differential of $\GNQ$ is the unique continuous linear endomorphism,
homogeneous of degree $1$, which satisfies the Leibniz rule and vanished except the following:
\[
  \diff \sum_{e\in Q_0} e^*  =  \sum_{a\in Q_1} \, [a,a^*].
\]
\end{itemize}

We will consider the following categories.
\begin{itemize}
  \item The bounded derived category $\DQ=\D^b(Q)$.
  \item The $m$-cluster category as the orbit category $\C_{m}(Q)=\DQ/\Sigma_m$ for $\Sigma_m=[m-1]\circ\tau^{-1}$, cf. \cite{Ke0}.
      Here, $\tau$ is the Auslander-Reiten functor on $\DQ$.
  \item By convention, let $\C_\infty(Q)=\per(Q)$ be the perfect derived category of $\k Q$, which happens to be triangle equivalent to $\DQ$.
  \item The perfect derived category $\per_N(Q)=\per(\GNQ)$ of $\GNQ$.
  \item The perfectly valued (=finite dimensional) derived category
  \begin{gather}\label{eq:D_NQ}
    \D_N(Q)=\pvd(\GNQ)
  \end{gather}
  of $\GNQ$.
\end{itemize}
It is well-known that $\C_{N-1}(Q)\cong\per(\GNQ)/\pvd(\GNQ)$.

\section{Fusion actions on triangulated categories}
\subsection{Fusion categories and actions}\

A \emph{fusion category} $\Fu$ is a finite, semisimple, abelian, rigid, monoidal category $(\Fu, \otimes, \mathbbm{1})$ with $\otimes$ being bilinear on morphisms and $\mathrm{End}_{\Fu}(\mathbbm{1})\cong\k$.
For any object $X\in\Fu$, denote by ${}^\star X$ and $X^\star$ its left and right duals respectively.
Without loss of generality, we assume that a fusion category is always strict.
Then the rigidity regarding the existence of left duals means that there exist morphisms
\[\mathrm{ev}_X\colon{}^\star X\otimes X \to \mathbbm{1}\quad \text{and}\quad
\mathrm{coev}_X\colon\mathbbm{1}\to X\otimes {}^\star X\]
such that
 \[(\mathrm{id}_X\otimes \mathrm{ev}_X)\circ(\mathrm{coev}_X\otimes \mathrm{id}_X)=\mathrm{id}_X\quad \text{and}\quad (\mathrm{ev}_X\otimes \mathrm{id}_{{}^\star X})\circ(\mathrm{id}_{{}^\star X}\otimes \mathrm{coev}_X)=\mathrm{id}_{{}^\star X}.\]
For the existence of right duals, one has similar morphisms and equalities.

\begin{example}
Let $G$ be a finite group.
There are two natural ways for associating fusion categories:
\begin{itemize}
  \item Let $\VG$ be the category of finite dimensional $G$-graded vector spaces.
  \item Let $\on{rep}_\k(G)$ be the category of finite dimensional representations of $G$ (over $\k$
  with $\on{char}\k\nmid|G|$).
\end{itemize}
\end{example}

Let $\KG(\Fu)$ denote the Grothendieck group of a fusion category $\Fu$. Then $\KG(\Fu)$ is a fusion ring (a unital based ring of finite rank)
with addition given by the direct sum and multiplication given by the tensor product.
By abuse of notation, we still write $x\in\KG(\Fu)$ for the isomorphism class of $x\in\Fu$.
The isomorphism classes of simple objects in $\Fu$ form a basis of $\KG(\Fu)$.

\begin{definition}\cite[Prop.~3.3.6]{EGNO}
Define \emph{Frobenius-Perron dimension}
\[
    \FP\colon \KG(\Fu) \to \RR,
\]
sending any $x\in\KG(\Fu)$ to the maximal non-negative eigenvalue of the matrix of left multiplication by $x$ on $\KG(\Fu)$.
Note that this map extends to $\CC$ and equips $\CC$ with a $\KG(\Fu)$-structure.
Denote by $\CC_{\Fu}$ the complex plane regarded as a $\KG(\Fu)$-module.
\end{definition}

\begin{example}
Let $q$ be a formal parameter and consider the fraction field of the ring 
of complex polynomials of $q$. The \emph{$m$-th quantum integer} is defined to be
\[
    [m]_q\colon= \frac{ q^{m}-q^{-m} }{q-q^{-1}}.
\]
When specializing $q=e^{i\pi/h}$ for $h\in\ZZ_+$, we obtain a complex number, denoted by $[m]_h$.

The \emph{Temperley-Lieb-Jones category} $\TLJ_h$ at $q=e^{i\pi/h}$ is a fusion category over $\CC$ generated by simple objects, denoted by their Frobenius-Perron dimensions (also known as the \emph{quantum dimensions})
\[\tlj{h}{1},\tlj{h}{2},\cdots,\tlj{h}{h-1},\]
with fusion rules given by
\[
    \tlj{h}i\otimes\tlj{h}j\cong
    \begin{cases}
        \tlj{h}{|i-j|+1}\oplus \tlj{h}{|i-j|+3}\oplus\cdots\oplus \tlj{h}{i+j-1},
            & \text{if $i+j\leq h$}; \\
        \tlj{h}{|i-j|+1}\oplus \tlj{h}{|i-j|+3}\oplus\cdots\oplus \tlj{h}{2h-(i+j)-1},
            &\text{if $i+j> h$}.
    \end{cases}
\]
Note that $(\tlj{h}{i})^{\star}\cong{}^\star\tlj{h}{i}\cong\tlj{h}{i}$ for all $i$.
\end{example}

\paragraph{\textbf{Fusion actions}}\

\begin{definition}
Let $\C$ be an additive category and $\Fu$ a fusion category.
We say there is a \emph{$\Fu$-action} on $\C$ provided that there is an additive monoidal functor $\Psi:\Fu\to\End(\C)$.
If $\C$ carries additional structures (e.g. abelian, triangulated, etc.),
objects in $\End(\C)$ are required to preserve such structures.

\end{definition}
For simplicity, we will let $x\mapsto \Psi_x, x\in\Fu$ and $f\mapsto \Psi_f, f\in \mathrm{Mor}_{\Fu}$ denote the corresponding $\Fu$-action.
Note that the $\Fu$-action naturally makes $\KG(\C)$ a $\KG(\Fu)$-module.

\subsection{Fusion-weighted species}\

Recall that a folding on a quiver $Q$ (or its underlying graph) is a group action $G$ on $Q$
(i.e. there is a group homomorphism $G\to\Aut Q$).
The folded quiver $\spec=Q/G$ is a quiver with
\begin{itemize}
  \item $G$-orbits of $Q_0$ as vertices,
  \item $G$-orbits of $Q_1$ as edges.
\end{itemize}
When the folding is nice,
say there is no double arrow in $\spec$ and each arrow $\bfa$ is covered by a subquiver $Q(\bfa)=G^{-1}(\bfa)$ of $Q$ of type ADE,
one gives the weight $h_{Q(\bfa)}$ to the edge $\bfa$,
where $h$ denotes the Coxeter number (associated to an ADE quiver).

\begin{example}[Folding of finite type quivers]\label{ex:folding}\
  \begin{itemize}
    \item
        $Q$ is of type $A_{2n-1}$ and $\spec$ is of type $C_n$ with $G=C_2$.
    \[
    \begin{tikzcd}[every arrow/.append style={no head},column sep=1pc,row sep=.1pc]
        &\bullet \ar[r]& \bullet \ar[r]& \cdots \ar[r]&\bullet\ar[dr]&  {}\ar[dd,red,<->,bend left=60]\\
        A_{2n-1}:&&&&&\circ      \\
        &\bullet \ar[r]& \bullet \ar[r]& \cdots \ar[r]&\bullet\ar[ur]&  {}
    \end{tikzcd}
    \qquad
    \begin{tikzcd}[every arrow/.append style={no head},column sep=1pc]\\
        C_n:&      \bullet \ar[r]& \bullet \ar[r]& \cdots \ar[r]&\bullet\ar[r,"4"]& \circ\\
    \end{tikzcd}
    \]
    \item
        $Q$ is of type $D_{n+1}$ and $\spec$ is of type $B_n$ with $G=C_2$.
    \[
    \begin{tikzcd}[every arrow/.append style={no head},column sep=1pc,row sep=.1pc]
        &&&&&\bullet \ar[dd,red,<->,bend left=60]\\
        D_{n+1}:&\circ \ar[r]& \circ \ar[r]& \cdots \ar[r]&\circ\ar[dr]\ar[ur]      \\
        &&&&&\bullet
    \end{tikzcd}
    \qquad
    \begin{tikzcd}[every arrow/.append style={no head},column sep=1pc]\\
        B_n:&      \circ \ar[r]& \circ \ar[r]& \cdots \ar[r]&\circ\ar[r,"4"]& \bullet\\
    \end{tikzcd}
    \]
    \item
        $Q$ is of type $E_6$ and $\spec$ is of type $F_4$ with $G=C_2$.
    \[
    \begin{tikzcd}[every arrow/.append style={no head},column sep=1pc,row sep=.1pc]
        &&&\bullet\ar[r]&\bullet \ar[dd,red,<->,bend left=60]\\
        E_{6}:&\circ \ar[r] & \circ \ar[dr]\ar[ur]      \\
        &&&\bullet\ar[r]&\bullet
    \end{tikzcd}
    \qquad
    \begin{tikzcd}[every arrow/.append style={no head},column sep=1pc]\\
        F_4:&      \circ \ar[r]& \circ \ar[r,"4"]& \bullet \ar[r]& \bullet\\
    \end{tikzcd}
    \]
    \item
        $Q$ is of type $D_4$ and $\spec$ is of type $G_2$ with $G=C_3$.
    \[
    \begin{tikzcd}[every arrow/.append style={no head},column sep=1pc,row sep=.1pc]
        &&\bullet \ar[dd,red,<->,bend left=60]\\
        D_{4}:&\circ \ar[r]\ar[dr]\ar[ur] & \bullet       \\
        &&\bullet
    \end{tikzcd}
    \qquad
    \begin{tikzcd}[every arrow/.append style={no head},column sep=1pc]\\
        G_2:&      \circ \ar[r,"6"]& \bullet\\
    \end{tikzcd}
    \]
  \end{itemize}
\end{example}

The representation and associated categories of folded quiver can be understood by \emph{species}.
We want to generalize such a notion to a fusion-weighted version.

\begin{definition}
An \emph{algebra} $A=(A,\mu_A,\eta_A)$ in a fusion category $\Fu$ is an object $A\in\Fu$ with morphisms
\[\mu_A:A\otimes A\to A\quad \text{and}\quad \eta_A:\mathbbm{1}\to A\]
such that
\begin{itemize}
\item $\mu_A\circ(\mathrm{id}_A\otimes \mu_A)=\mu_A\circ(\mu_A\otimes\mathrm{id}_A),$
\item $\mu_A\circ(\eta_A\otimes \mathrm{id}_A)=\mathrm{id}_A=\mu_A\circ(\mathrm{id}_A\otimes \eta_A).$
\end{itemize}


A right $A$-module $M=(M,\mu_{M,A})$ in $\Fu$ is an object $M\in\Fu$ with morphism
\[\mu_{M,A}:M\otimes A\to M\]
such that
\[\mu_{M,A}\circ(\mathrm{id}_M\otimes \mu_{A})=\mu_{M,A}\circ(\mu_{M,A}\otimes \mathrm{id}_A)\quad\text{and}\quad\mu_{M,A}\circ(\mathrm{id}_M\otimes\eta_A)=\mathrm{id}_M.\]

Denote by $\mod_{\Fu}(A)$ the category of all right $A$-modules in $\Fu$. Note that $\mod_{\Fu}(A)$ is semisimple and denote by $\Sim_{\Fu} A$ the set of (isomorphism classes of) simple objects in $\mod_{\Fu}(A)$.

A left $A$-module $M=(M,\mu_{A,M})$ in $\Fu$ is an object $M\in\Fu$ with morphism
\[\mu_{A,M}:A\otimes M\to M\]
such that
\[\mu_{A,M}\circ(\mathrm{id}_A\otimes \mu_{A,M})=\mu_{A,M}\circ(\mu_{A}\otimes \mathrm{id}_M)\quad\text{and}\quad\mu_{A,M}\circ(\eta_A\otimes\mathrm{id}_M)=\mathrm{id}_M.\]
Let $A=(A,\mu_A,\eta_A)$ and $B=(B,\mu_B,\eta_B)$ be two algebras in $\Fu$. An $A$-$B$-bimodule $M=(M,\mu_{A,M},
\mu_{M,B})$ in $\Fu$ is an object $M\in\Fu$ such that $(M,\mu_{A,M})$ is a left $A$-module and
$(M,\mu_{M,B})$ is a right $B$-module satisfying that
\[\mu_{A,M}\circ(\mathrm{id}_A\otimes \mu_{M,B})=\mu_{M,B}\circ(\mu_{A,M}\otimes \mathrm{id}_B).\]
\end{definition}

For a right $A$-module $M=(M,\mu_{M,A})$ and a left $A$-module $N=(N,\mu_{A,N})$ in $\Fu$,
define the tensor product $M\otimes_{A}N$ to be the kernel of the morphism
\[\mu_{M,A}\otimes \mathrm{id}_N-\mathrm{id}_M\otimes \mu_{A,N}\colon M\otimes A\otimes N\to M\otimes N.\]


For any algebra $A$ and any object $L$ in $\Fu$, each $L\otimes{}_-$ can be viewed as an endofunctor of any $\mod_{\Fu}(A)$. Let $A,B$ be two algebras and $M$ an $A$-$B$-bimodule in $\Fu$. Then the functor
\[{}_-\otimes_{A}M\colon\mod_{\Fu}(A)\to\mod_{\Fu}(B)\]
commutes with all functors $L\otimes{}_-, L\in\on{Obj}(\Fu)$ up to isomorphism. Since $\Fu$ is fusion category, then $L\otimes{}_-$ is an exact endofunctor of $\mod_{\Fu}(A)$. Note that for any right $A$-module $X$ we have the exact sequence
\[\begin{tikzcd}[column sep=3pc]
L\otimes (X\otimes_{A} M)
    \arrow[hookrightarrow, r,"\mathrm{id}_L\otimes\varphi"]  &L\otimes X\otimes A\otimes M \arrow[rrrr,"\mathrm{id}_L\otimes\mu_{X,A}\otimes \mathrm{id}_M-\mathrm{id}_L\otimes\mathrm{id}_X\otimes \mu_{A,M}"]
    &&&&L\otimes X\otimes M,
\end{tikzcd}\]
where $X\otimes_{A} M$ together with $\varphi$ is the kernel of $\mu_{X,A}\otimes \mathrm{id}_M-\mathrm{id}_X\otimes \mu_{A,M}$. By the universal property, there exists a unique isomorphism
\[\phi_{L,X,M}:(L\otimes X)\otimes_{A} M\to L\otimes (X\otimes_{A} M).\]
%

\begin{definition}
A \emph{$\Fu$-weighted specie} (or $\Fu$-specie for short) $\spec=(\spec,\wt)$ consists of a quiver $\spec$ and
a $\Fu$-valued function on $\spec$
\begin{gather}\label{eq:wt}
    \wt\colon \spec\to\on{Obj}(\Fu)
\end{gather}
such that
\begin{itemize}
  \item for any $\bfi\in\spec_0$, each $\wt(\bfi)$ is a semisimple algebra in $\Fu$;
  \item for any $\bfa:\bfi\to\bfj$ in $\spec_1$, each $\wt(\bfa)$ is a $\wt(\bfi)$-$\wt(\bfj)$-bimodule.
\end{itemize}
\end{definition}
Note that each $\wt(\bfa)$ gives an additive functor
\[
    \WT_\bfa={}_-\otimes_{\wt(\bfi)}\wt(\bfa)\colon \mod_{\Fu}(\wt(\bfi))\to\mod_{\Fu}(\wt(\bfj)).
\]
Let $s$ and $t$ be the source and target functions of $\spec$. For any path $\bfp=\bfa_n\cdots\bfa_2\bfa_1$, where each $\bfa_i$ lies in $\spec_1$, we can define the functor \[\WT_\bfp\colon=\WT_{\bfa_n}\circ\left(\circ\cdots\left(\WT_{\bfa_2}\circ\WT_{\bfa_1}\right)\right)\colon\mod_{\Fu}(\wt(s(\bfa_1)))\to\mod_{\Fu}(\wt(t(\bfa_n)).\]
For each $\bfi\in\spec_0$, denote by $\mathbf{e}_\bfi$ the corresponding trivial path. Set $\WT_{\mathbf{e}_\bfi}$ to be the identity functor of $\mod_{\Fu}(\wt(\bfi))$. By definition, it is clear that $\WT_{\bfp_2}\circ\WT_{\bfp_1}=\WT_{\bfp_2\bfp_1}$ whenever the path $\bfp_2\bfp_1$ makes sense.

\begin{definition}
A \emph{representation} $\wtV$ of a $\Fu$-specie $\spec=(\spec,\wt)$ consists of
\begin{itemize}
    \item for each $\bfi\in\spec_0$, an object $\wtV_{\bfi}\in \mod_{\Fu}(\wt(\bfi))$ and
    \item for each $\bfa:\bfi\to\bfj$ in $\spec_1$, a morphism $\wtV_{\bfa}:\wtV_{\bfi}\otimes_{\wt(\bfi)} \wt(\bfa)\to \wtV_{\bfj}$ in $\mod_{\Fu}(\wt(\bfj))$.
\end{itemize}
A morphism $\wtf:\wtV\to \wtW$ of representations of a $\Fu$-specie $\spec=(\spec,\wt)$ is a collection of morphisms
\[\wtf_{\bfi}:\wtV_{\bfi}\to \wtW_{\bfi}\]
in $\mod_{\Fu}(\wt(\bfi))$, where $\bfi\in\spec_0$, such that the diagram
\[
\begin{tikzcd}[column sep=large]
\wtV_{\bfi}\otimes_{\wt(\bfi)} \wt(\bfa)
    \arrow[r,"\wtV_{\bfa}"] \arrow[d,"\wtf_{\bfi}\otimes_{\wt(\bfi)}\mathrm{id}_{\wt(\bfa)}"'] &\wtV_{\bfj}\arrow[d,"\wtf_{\bfj}"]  \\
\wtW_{\bfi}\otimes_{\wt(\bfi)} \wt(\bfa)\arrow[r,"\wtW_{\bfa}"]& \wtW_{\bfj}
\end{tikzcd}
\]
commutes for each $\bfa:\bfi\to\bfj$ in $\spec_1$.
\end{definition}
All representations of a $\Fu$-specie $\spec=(\spec,\wt)$ together with morphisms between those representations form a category, denoted by $\mathrm{rep}(\spec)$. The direct sum, kernel, cokernel in $\mathrm{rep}(\spec)$ can be defined component-wisely. Therefore, the category $\mathrm{rep}(\spec)$ is abelian.

There is a $\Fu$-action on $\mathrm{rep}(\spec)$ (viewed as a left module category over $\Fu$), i.e.  there exists an additive monoidal functor
\[\Psi:\Fu\to\mathcal{E}\text{nd}(\mathrm{rep}(\spec)),\]
which sends any object $X$ in $\Fu$ to the endofunctor $X\otimes{}_-$ of $\mathrm{rep}(\spec)$ and any morphism $f:X\to Y$ in $\Fu$ to the natural transformation $f\otimes{}_-\colon X\otimes{}_-\to Y\otimes{}_-$. In fact, the functor $X\otimes{}_-$ is given by sending any object $\wtV$ to $X\otimes\wtV$ with
\[
    (X\otimes\wtV)_{\bfi}\colon= X\otimes\wtV_{\bfi}\quad\text{and}\quad
        (X\otimes\wtV)_{\bfa}\colon=
            (\mathrm{id}_{X}\otimes\wtV_{\bfa})\circ\phi_{X,V_{s(\bfa)},\wt(\bfa)}
\]
for any $\bfi\in\spec_0$ and $\bfa\in\spec_1$,
and sending any morphism $f:\wtV\to\wtW$ to $X\otimes f:X\otimes\wtV\to X\otimes\wtW$ with
\[(X\otimes f)_{\bfi}\colon=\mathrm{id}_{X}\otimes f_{\bfi}\]
for any $\bfi\in\spec_0.$

\begin{definition}\label{def:unfolding}
The \emph{unfolded quiver} $Q=\unfold(\spec,\wt)$ of a $\Fu$-specie $(\spec,\wt)$ is defined to be a quiver which
\begin{itemize}
  \item has vertex set given by $\bigsqcup_{\bfi\in\spec_0} \Sim_{\Fu}\wt(\bfi)$, and
  \item for each $\bfa:\bfi\to\bfj$ in $\spec_1$, $m_{\bfa,X,Y}$ arrows from $X\in\Sim_{\Fu}\wt(\bfi)$ to $Y\in\Sim_{\Fu}\wt(\bfj)$, where
    \[  m_{\bfa,X,Y}=\dim\Hom_{\mod_{\Fu}(\wt(\bfj))}( X\otimes_{\wt(\bfi)}\wt(\bfa),Y).\]
\end{itemize}
We call
\begin{gather}\label{eq:f:Q2Q}
    \wtF\colon Q \xrightarrow{/\Fu_\wt} \spec,
\end{gather}
a $\Fu$-weighted folding.
\end{definition}
Similar to the proof of \cite[Theorem~3.1]{H1}, for any $\Fu$-specie $\spec=(\spec,\wt)$
one can define an exact functor
\[\Theta:\mathrm{rep}(\spec)\to \mathrm{rep}(Q),\]
and show that $\Theta$ gives an equivalence of abelian categories.



\begin{example}[Group action as fusion action]\
Let $G$ be a finite abelian group and $\Fu=\VG$.
We will use $g\in G$ to denote the corresponding simple object in $\VG$.
Then a $G$-action on $Q$ with $G$-orbit $\spec$ can be regarded as a $\Fu$-weighted folding as follows.
\begin{itemize}
  \item For each $\bfi\in\spec$, let $\wt(\bfi)=\bigoplus_{g\in \on{stab}(\bfi) } g$ and then
    $\mod_{\Fu} (\wt(\bfi))= \on{\mathsf{vec}}^{G/\on{stab}(\bfi)}(\k)$, where $\on{stab}(\bfi)$ denotes the (pointwise) stabilizer of the orbit $\bfi$.
  \item For each $\bfa\colon\bfi\to\bfj$, let $\wt(\bfa)=\wt(\bfi)\otimes_{X_{\bfi\bfj}}\wt(\bfj)$,
  where $X_{\bfi\bfj}=\bigoplus_{g\in \on{stab}(\bfa)} g$ and $\on{stab}(\bfa)$ denotes the (pointwise) stabilizer of the orbit $\bfa$. Note that $\on{stab}(\bfa)=\on{stab}(\bfi)\bigcap\on{stab}(\bfj)$.
    Hence $\wt(\bfa)$ is a $\wt(\bfi)$-$\wt(\bfj)$-bimodule.
\end{itemize}
Then the unfolded quiver is $Q$.

For instance, take the folding from $D_4$ to $G_2\colon \2\xrightarrow{\;\bfa\;} \1$ in \Cref{ex:folding}.
We have $G=C_3=\{1,\omega,\omega^2\}$ for $w=e^{2\bfi\pi/3}$,
$\wt(\1)=1\oplus\omega\oplus\omega^2$, $\wt(\mathbf{2})= 1$ and $\wt(\bfa)=1\oplus\omega\oplus\omega^2$ as a $\wt(\2)$-$\wt(\1)$-bimodule.
The unfolded quiver is a $D_4$ quiver:
\[\begin{tikzcd}
    1\ar[dr]\\ \omega \ar[r] & 1+\omega+\omega^2 \\ \omega^2\ar[ur]
\end{tikzcd}\]
\end{example}

\begin{example}\label{ex:H1}[TLJ-weighted species, \cite{H1}]\

Let $\spec$ be an edge-valued quiver with $\vv\colon \spec_1\to\ZZ_{\ge3}$.
Define
\[
    \TLJ_\vv\colon=\bigboxtimes_{3<h\in\on{im}(\vv)} \TLJ_h,
\]
where $\boxtimes$ means taking the Deligne's tensor product (i.e. the fusion version of direct product, cf. \cite[\S~1.4]{H1}).
The tensor product in $\TLJ_\vv$ is defined component wise. The category $\TLJ_\vv$ is also a fusion category.
Denote by $\tlj{h}{k}^\vv$ the simple object in $\TLJ_\vv$ with $\tlj{h}{k}$ in the $\TLJ_h$ component and $\tlj{?}{1}$ in all other components, where $?\in\on{im}(\vv)\setminus\{h\}$.
The unit for $\TLJ_\vv$ is $\mathbbm{1}$ with $\tlj{h}{1}$ in the $\TLJ_h$ component for any $3<h\in\on{im}(\vv)$.
It is clear that $\mathbbm{1}$ is a simple algebra and $\mod_{\Fu}(\mathbbm{1})=\TLJ_\vv$.

Define a $\TLJ_\vv$-valued function $\wt\colon\spec\to \on{Obj}(\TLJ_\vv)$ on $\spec$ by
\[\begin{array}{rcl}
    \bfi\in\spec_0 & \mapsto & \mathbbm{1},\\
    \bfa\in\spec_1 & \mapsto & \begin{cases}
                                 \tlj{\vv(\bfa)}{2}^\vv, & \mbox{if $\vv(\bfa)>3$,}  \\
                                 \mathbbm{1}, & \mbox{otherwise}.
                               \end{cases}
  \end{array}
\]
Then $(\spec, \wt)$ is a $\TLJ_\vv$-weighted specie.
\end{example}
Note that our convention is slightly different from \cite{H1}. cf \S~6.2 there.


\begin{example}[Weighted (un)folding of finite type quivers]\label{ex:w-folding}\

We give a list of weighted foldings from ADE quivers to Coxeter-Dynkin quivers
that we will study later, cf. \cite{DT1,DT2}.

In \Cref{ex:H1}, when $\spec$ is a finite type Coxeter diagram,
the unfolded $Q$ are shown in \Cref{fig:BCEI} and \Cref{fig:H234}.

\begin{figure}[htbp]\centering
\begin{tikzpicture}
\foreach \j in {0,-1} {
\draw[thick] (3.6,\j) -- (4.8,\j-1)
            (4.8,\j) -- (3.6,\j-1);}
\draw[] (2.4,1)node{$A_{2\bfn-1}\bigsqcup D_{\bfn+1}$}
        (2.4,-4.6)node{$B_{\bfn}=C_{\bfn}$};
\draw[ultra thick] (2.8,-4) -- (4.8,-4)
                (0,-4) -- (2,-4);
\foreach \j in {0,-1,-2,-4} {
    \draw[thick] (2.4,\j) node {$\cdots$}
                (0,\j) -- (1.9,\j)
                (2.8,\j) -- (3.6,\j);
    \foreach \i in {0,1,3,4} {
        \filldraw[white,fill=white] (1.2*\i,\j) circle (0.15);\filldraw[fill=white] (1.2*\i,\j) \ww;
    }}
\draw[very thick,->,>=stealth,red]
    (2.4,-2.4) to node[left] {$f$} (2.4,-3.2);
\draw (4.2,-3.6) node {$4$};
\begin{scope}[shift={(6,0)}]
\draw[] (1.8,1)node{$E_6^{\bigsqcup2}$}
        (1.8,-4.6)node{$F_4$};
\foreach \j in {0,-1} {
    \draw[thick] (1.2,\j) -- (2.4,\j-1);
    \draw[thick] (2.4,\j) -- (1.2,\j-1);}
\draw[ultra thick] (0,-4) -- (3.6,-4);
\foreach \j in {0,-1,-2,-4} {
    \draw[thick] (0,\j) -- (1.2,\j);
    \draw[thick] (2.4,\j) -- (3.6,\j);
    \foreach \i in {0,1,2,3} {
        \filldraw[white,fill=white] (1.2*\i,\j) circle (0.15);\draw[](1.2*\i,\j) \ww;
    }}
\draw[very thick,->,>=stealth,red]
    (1.8,-2.4) to node[left] {$f$} (1.8,-3.2);
\draw[thick] (1.8,-3.7)node {$4$};
\end{scope}
\end{tikzpicture}
\caption{Weighted folding for type $B_{\bfn}/C_{\bfn}$ and $F_4$}
\label{fig:BCEI}
\end{figure}
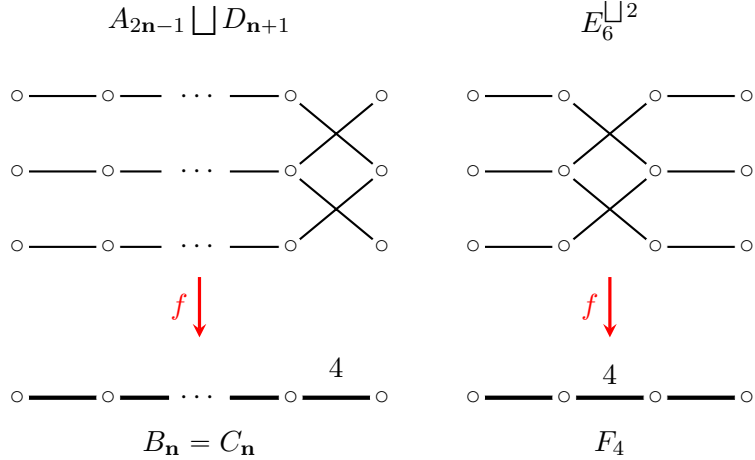

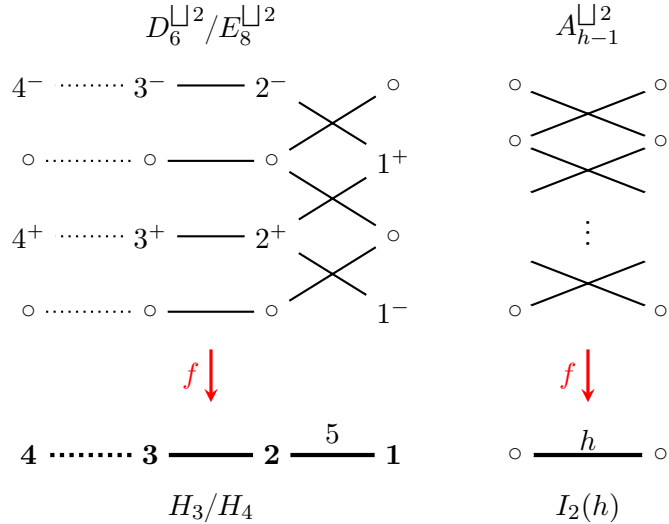
\begin{figure}[htpb]\centering
\begin{tikzpicture}[yscale=1,xscale=1.6]
\foreach \j in {2,3,4}{
  \draw[] (-\j,0) node (x\j) {$\j^+$}
    (-\j,2) node (y\j) {$\j^-$}
    (-\j,-2.9) node (z\j) {$\mathbf{\j}$};
}  \draw(-1,-1) node (x1) {$1^-$}
    (-1,1) node (y1) {$1^+$}
    (-1,-2.9) node (z1) {$\1$};

\foreach \j in {2,3,4}{
  \draw[] (-\j,1) node (t\j) {$\circ$}
    (-\j,-1) node (s\j) {$\circ$};
}  \draw(-1,0) node (s1) {$\circ$}
    (-1,2) node (t1) {$\circ$};

\draw[dotted,thick] (x3)to(x4) (y3)to(y4) (s4)to(s3) (t4)to(t3);
\draw[dotted,ultra thick] (z3)to(z4);
\draw[thick] (x1)to(x2)to(y1)to(y2) (x2)to(x3) (y2)to(y3) (s2)to(s3) (t2)to(t3)
    (t1)to(t2)to(s1)to(s2);
\draw[ultra thick](z1)to node[above]{$5$}(z2)to(z3);
\draw[very thick,->,>=stealth,red](-2.5,-1.5) to node[left] {$f$} (-2.5,-2.2);
\draw(-2.5,2.8) node {$D_6^{\bigsqcup2}/E_8^{\bigsqcup2}$}
    (-2.5,-3.6)node{$H_3/H_4$};

\begin{scope}
\draw[] (.6,2.8)node{$A_{h-1}^{\bigsqcup2}$}
        (.6,-3.6)node{$I_2(h)$};
\foreach \j in {2,1.25,-0.25} {
    \draw[thick] (0,\j) -- (1.2,\j-0.75)
            (1.2,\j) -- (0,\j-0.75);}
\draw[ultra thick] (0,-2.9) -- (1.2,-2.9)
            (0.6,0.15) node {$\vdots$};
\foreach \j in {2,1.25,-1,-2.9} {
    \foreach \i in {0,1} {
        \filldraw[white,fill=white] (1.2*\i,\j) circle (0.15);\draw[](1.2*\i,\j) \ww;
    }}
\foreach \j in {0.5,-0.25} {
    \foreach \i in {0,1} {
        \filldraw[white,fill=white] (1.2*\i,\j) circle (0.15);\draw[](1.2*\i,\j) node[white]{$\bullet$}node[white]{$\circ$};
    }}
\draw[very thick,->,>=stealth,red]
    (0.6,-1.5) to node[left] {$f$} (0.6,-2.2);
\draw[thick] (0.6,-2.7) node {$h$};
\end{scope}
\end{tikzpicture}
\caption{Weighted folding type $H_3/H_4$ and $I_2(h)$}
\label{fig:H234}
\end{figure}

\end{example}

\subsection{Functorial action}\label{sec:fun.action}
\paragraph{\textbf{Induced actions for quivery categories}}\

Let $\Fu(\spec)$ denote the additive category whose objects are the same as in $\bigsqcup_{\bfi\in\spec_0}\mod_{\Fu}(\wt(\bfi))$ and whose morphisms are given by
\begin{equation}\label{eq:Hom1}
  \Hom_{\Fu(\spec)}(X, Y)\colon=
  \bigoplus_{{\text{path }}\bfp\colon\bfi\rightsquigarrow\bfj}
    \Hom_{\mod_{\Fu}(\wt(\bfj))} ( \WT_\bfp(X), Y)
\end{equation}
for indecomposable objects $X\in\mod_{\Fu}(\wt(\bfi))$ and $Y\in\mod_{\Fu}(\wt(\bfj))$ with
composition defined component wise in the following way:
for any
$f\in\Hom_{\mod_{\Fu}(\wt(\bfj))} ( \WT_{\bfp_1}(X), Y)$ and
$g\in\Hom_{\mod_{\Fu}(\wt(\bfl))} ( \WT_{\bfp_2}(Y), Z)$ with $\bfp_1$ being a path from $\bfi$ to $\bfj$ and $\bfp_2$ a path from $\bfk$ to $\bfl$,
we define
\[g\diamond f\colon=\begin{cases} g\circ\WT_{\bfp_2}(f)\in \Hom_{\mod_{\Fu}(\wt(\bfl))} ( \WT_{\bfp_2\bfp_1}(X), Z),
& \text{if $\bfj=\bfk$}; \\
0, &\text{otherwise}.
\end{cases}
\]

One sees that $\Fu(\spec)$ is in fact equivalent to $\on{Proj} \k Q$ as additive categories.
Note that there is a $\Fu$-action on $\Fu(\spec)$, i.e.,  there is an additive monoidal functor
\begin{equation}\label{eq:Psi}
    \Psi:\Fu\to\mathcal{E}\text{nd}(\Fu(\spec))
\end{equation}
by left tensoring.
Then we have
\begin{equation}\label{eq:rep Q}
    \rep(Q)=\on{Fun}(\Fu(\spec),\on{Vect}({\k}) ),
\end{equation}
which admits an induced $\Fu$-action, also denoted by $\Psi$.
Moreover, by \cite[Thm.~8.15]{Toen} (cf. \cite[Thm.~4.6]{Ke2}),
\begin{equation}\label{eq:DbQ}
    \D^b(Q)=\on{rep}(\Fu(\spec),\D^b(\k)).
\end{equation}
Hence, we obtain an induced $\Fu$-action.

Finally, as the $\Fu$-action commutes with $[1]$ and $\tau$,
we also have the $\Fu$-action on $\C_m(Q)$.

\paragraph{\textbf{Induced actions for Calabi-Yau completion}}\

Next, we construct the Calabi-Yau double $\CYGQ$ of $\Fu(\spec)$ following \cite{Ke1}.
Fix an integer $N$ and assume that $\spec$ (i.e. its arrow set) is $\ZZ$-graded, with degree function $\deg$.
That is, the category $\CYGQ$ with the same object set $\Fu(\spec)$ and
\begin{equation}\label{}
  \Hom_{ \CYGQ }(X, Y)\colon=\Hom_{\Fu(\spec)}(X, Y)\oplus\Hom_{\Fu(\spec)}(Y, X)^*
\end{equation}
for indecomposable objects $X$ and $Y$.
Now, we make $\CYGQ$ a dg category:
\begin{itemize}
  \item the morphisms in \eqref{eq:Hom1} inherit the degree of the path $\bfp$;
  \item if $f\in \Hom_{\Fu(\spec)}(X, Y)$ is not identity, then $\deg f^*=2-N-\deg f$;
  \item $\deg(\id_X^*)=1-N$;
  \item the degree of the differential $\diff$ is 1;
  \item the nontrivial differential is given by
  \[
    \diff(\id_X^*)=\sum_{ f } [f,f^*],
  \]
  where the sum is over a basis of \eqref{eq:Hom1}.
\end{itemize}

The additive category $\CYGQ$ is equivalent to $\on{Proj}\GNQ$, cf. \Cref{sec:quivery}.
Hence as \eqref{eq:DbQ} we have,
\begin{equation}\label{eq:pvd}
    \pvd(\GNQ)=\on{rep}_{}\left(\CYGQ,\D^b(\k)\right)
\end{equation}
with an induced $\Fu$-action, which is $\D_N(Q)$ in \eqref{eq:D_NQ}.
There is also an induced $\Fu$-action on $\per(\GNQ)$
as it consists of twisted objects over $\on{Proj}\GNQ$.

In summary, there is a $\Fu$-action, denoted by $\Psi$,
on various categories associated to $Q=\unfold(\spec,\wt)$.

\section{Fusion-stable structures}
\subsection{Fusion-stable (or fusion-invariant) structures}\

Recall that our setting is that
let $\C$ be an additive category with a fusion action $\Fu$, i.e. there is an additive monoidal functor $\Psi:\Fu\to\mathcal{E}\mathrm{nd}(\C)$.

An additive subcategory $\C_0\subset\C$ is $\Fu$-stable if the functor $\Psi:\Fu\to\mathcal{E}\mathrm{nd}(\C)$ can be restricted to a functor $\Psi:\Fu\to\mathcal{E}\mathrm{nd}(\C_0)$, where the latter is automatically additive and monoidal.
Equivalently, for any $x\in\Fu$ and $C\in\C_0$, the object $\Psi_x(C)\in\C_0$.

\begin{definition}
Suppose that $\C$ is an abelian or a triangulated category with a fusion action $\Fu$
(Note that $\Psi_x$ is required to be exact, for any $x\in\Fu$, in this case).
A torsion pair $(\hT,\hF)$ on $\C$ is $\Fu$-stable if both $\hT$ and $\hF$ are $\Fu$-stable.
Similarly for other tilting structures, i.e.
\begin{itemize}
  \item $\Fu$-stable t-structures,
  \item $\Fu$-stable hearts,
  \item $\Fu$-stable $m$-CTS.
\end{itemize}
\end{definition}

A forward tilting $\mu^\sharp_\hF\colon\h\to\h^\sharp$ between two $\Fu$-stable hearts
must be with respect to a $\Fu$-stable torsion pair,
as $\hT=\h^\sharp\cap\h$ and $\hF=\h^\sharp[-1]\cap\h$.
Similarly, a forward mutation $\mu_{\hY\setminus\hX}^\sharp\colon\hY\to\hY^\sharp$
between two $\Fu$-stable $m$-CTS must be with respect to a $\Fu$-stable subcategory $\hY\setminus\hX$ of $\hY$.

For any object $M$ of $\C$, denote by $\Fu(M)$ the $\Fu$-stable additive subcategory generated by all direct summands of $\Psi_x(M)$, for all $x\in\Fu$, called the $\Fu$-orbit of $M$.
Similarly when $M$ is a set of objects or a subcategory.

\begin{definition}
A $\Fu$-indecomposable ($\Fu$-ind for short) subcategory $\C_0\in\C$ is the $\Fu$-orbit of some indecomposable object in $\C$.
When $\C$ is abelian, we say a $\Fu$-ind subcategory is $\Fu$-simple if it is the $\Fu$-orbit of some simple in $\C$.
\end{definition}

\paragraph{\textbf{Exchange graphs}}\

\begin{definition}
Let $\D$ be a triangulated category with a fusion action $\Fu$.
\begin{itemize}
  \item A forward tilting $\mu^\sharp_\hF\colon\h\to\h^\sharp$ is called $\Fu$-simple if
    $\hF=\Fu(S)$ is $\Fu$-simple in $\h$.
  \item The $\Fu$-stable exchange graph of hearts of $\D$ is the graph whose
  vertices are $\Fu$-stable hearts and whose edges are $\Fu$-simple forward tilting.
  Denote it by $\EG(\D)^\Fu$.
  \item A forward $m$-cluster mutation $\mu_{\hY\setminus\hX}^\sharp\colon\hY\to\hY^\sharp$
  is $\Fu$-ind, if $\hY\setminus\hX=\Fu(X)$ is a $\Fu$-ind subcategory.
  \item The $\Fu$-stable $m$-cluster exchange graph is the graph whose vertices are $\Fu$-stable $m$-CTS
  and whose arrows are $\Fu$-ind mutations. Denote it by $\CEG_m(\D)^\Fu$.
\end{itemize}
As before, we use silting for $\infty$-cluster.
\end{definition}

\begin{remark}
Note that $\EG(\D)^\Fu$ is a subset of $\EG(\D)$ but NOT a subgraph.
The edges in $\EG(\D)^\Fu$ correspond to certain paths in $\EG(\D)$.
For instance, \cite[Fig.~1]{CQ} demonstrate $\EG(C_2)$ (vertices underlined) sitting inside $\EG(A_3)$.
The edges of $\EG(C_2)$ correspond to directed paths in the oriented associahedron $\EG(A_3)$.
Similarly for $\CEG_m(\D)^\Fu$.
\end{remark}

\begin{lemma}\label{lem:G-simple}
Let $\mu^\sharp_{\Fu(S)}\colon\h\to\h^\sharp_{\Fu(S)}$ be a $\Fu$-simple forward tilting.
If $\Fu(S)$ is rigid (i.e. $\Ext^1$-vanishing) and the indecomposable objects in $\Fu(S)$
are simples $S_1,\ldots,S_l$ in $\h$,
then $\mu^\sharp_{S_i}$ are pairwise commutative, for $1\le i\le l$,
and
\begin{equation}\label{eq:G-simple}
    \h^\sharp_{\Fu(S)}= \mu^\sharp_{\Fu(S)}(\h)= \big( \prod_{i=1}^l \mu^\sharp_{S_i} \big) (\h).
\end{equation}
\end{lemma}
\begin{proof}
This follows from repeatedly using the square/commutative relation between two simple tilting
(see \cite[Lem.~6.1]{QW} or \cite[Rem.~2.13]{KQ2}).
Equivalently, the tilting in \eqref{eq:G-simple} is with respect to simultaneously tilting all simples in $\Fu(S)$,
or the torsion free class that equals $\Add\Fu(S)$.
\end{proof}

By simple-projective duality, we have the following in a similar way.
\begin{lemma}\label{lem:G-Ind}
Let $\mu^\sharp_{\Fu(Z)}\colon\hY\to\hY^\sharp_{\Fu(Z)}$ be a $\Fu$-ind forward mutation.
Denote by $\{Z_1,\ldots,Z_l\}$ the indecomposable objects in $\Fu(Z)$.
If $\Fu(Z)$ is discrete, i.e. $\Hom(Z_i,Z_j)$=0 for any $i\ne j$),
then $\mu^\sharp_{Z_i}$ are pairwise commutative, for $1\le i\le l$,
and
\[
    \hY^\sharp_{\Fu(Z)}= \mu^\sharp_{\Fu(S)}(\hY)= \big( \prod_{i=1}^l \mu^\sharp_{Z_i} \big) (\hY).
\]

\end{lemma}

\subsection{Fusion-stable tilting theory for weighted acyclic quivers}\

Similar to group-stable case in \cite{CQ},
one can easily generalize results on tilting theory to fusion-stable version.
For instance, we state the generalization of results in \cite{KQ1} without much of effort.

\begin{theorem}\label{thm:KQ}
Suppose we have a $\Fu$-weighted folding
$\wtF\colon Q \xrightarrow{/\Fu_\wt} \spec$ and $Q$ is acyclic.
By taking the $\Fu$-stable part, we have the following.
\begin{itemize}
  \item $N\in\ZZ_{\ge2}$, the simple-projective duality holds, i.e.
    \begin{gather}
        \EGp(\D_N(Q))^\Fu\cong\SEGp(\per_N(Q))^\Fu.
    \end{gather}
    \item For $m\in\ZZ_{\ge1}$, $\EGp(\D_{m+1}(Q))^\Fu$ is a covering of $\CEG_{m}(Q)^\Fu$.
  \item The $2$-cluster exchange graph $\CEG_2(Q)^\Fu$ is exactly
  the double of the interval $\EGp[\h,\h[1]]^\Fu$ in $\EGp(\D_{3}(Q))^\Fu$
  for any $\h\in\EGp(\D_{3}(Q))^\Fu$ , i.e. by adding an opposite edge for each existing edge.
\end{itemize}
\end{theorem}

\subsection{Fusion-stable stability conditions}\

Let $\C$ be an abelian or a triangulated category with a fusion action $\Fu$. Then the associated additive, exact, monoidal functor $\Psi:\Fu\to\mathcal{E}\mathrm{nd}(\C)$ can be decategorified to a ring homomorphism
\[\KG(\Psi):\KG(\Fu)\to \mathrm{End}(\KG(\C)),\]
that makes $\KG(\C)$ a $\KG(\Fu)$-module.

\begin{definition}
A stability condition $\sigma=(Z,\hP_\RR)$ on a triangulated category $\D$ is $\Fu$-stable if
\begin{itemize}
  \item the slicing $\hP_\RR$ (i.e. each of its slice $\hP(\phi))$ is $\Fu$-stable and
  \item the central charge is $\Fu$-stable, i.e.
  $$Z\in\Hom_{\ZZ}(\KG(\D),\CC)^\Fu\colon=\Hom_{\KG(\Fu)}(\KG(\D),\CC_{\Fu}).$$
\end{itemize}
\end{definition}
Note that
\[
    \Hom_{\ZZ}(\KG(\D),\CC)^\Fu:=\{Z\in\Hom_{\ZZ}(\KG(\D),\CC)\mid
        Z([\Psi_x(Y)])=\FP(x)\cdot Z([Y]),\forall x\in \Fu\}
\]

We generalize the main result (Thm.~1.2) in \cite{B1} to the fusion-stable version,
to demonstrate our philosophy in the introduction.
Such a result is indeed expected by Heng \cite[Remark.~3.3.2]{H2}.
Note that similar result has been considered in \cite[Thm.~3.10]{IQ},
where the deformation theorem in \cite{B1} plays a crucial role in both proofs.

Recall two notions from \cite[\S~4]{B1}.
\begin{itemize}
  \item A \emph{quasi-abelian category} is an additive category with kernels and cokernels such that every pullback of a strict epimorphism is a strict epimorphism, and every pushout of a strict monomorphism is a strict monomorphism. (A morphism $f$ is \emph{strict} if the canonical map $\mathrm{coim}f\to\mathrm{im}f$ is an isomorphism.)
  \item A \emph{skewed stability function} on a quasi-abelian category $\hua{A}$ is a group homomorphism $Z\in\Hom_\ZZ( \KG(\hua{A}),\CC )$ such that there is a tilted half-plane
      $\mathbb{H}_\alpha=e^{\bfi \pi \alpha}\cdot\mathbb{H}\subset\CC$
      for some $\alpha\in\RR$ satisfying $Z(E)\in\mathbb{H}_\alpha$ for any object $0\ne E\in\hua{A}$.
\end{itemize}

\begin{theorem}\label{thm:stab}
Let $\D$ be a triangulated category with a fusion $\Fu$-action.
$\Stab(\D)^\Fu$ is a complex manifold with local coordinate
$Z\in\Hom_{\KG(\Fu)}(\KG(\D),\CC_{\Fu})$.
\end{theorem}
\begin{proof}
We first recall some details of \cite[Thm.~7.1]{B1}.

Let $\epsilon$ be a small positive real number, say in $(0,1/8)$.
Given a stability condition $\sigma=(Z,\hP)\in\Stab(\D)$.
Let $W\in\Hom_{\ZZ}(\KG(\D),\CC)$ be a central charge satisfying
\begin{equation}\label{eq:W}
    | W(M)-Z(M) |<\sin(\pi\epsilon)|Z(M)|
\end{equation}
for any $M$ in $\D$.
By \cite[Thm.~7.1 and Lem.~6.4]{B1}, there is a unique stability condition $\varsigma$ on $\D$ with central charge $W$ satisfying $d(\sigma,\varsigma)<\epsilon$, where the distance function on $\Stab$ is defined as in \cite[\S6]{B1}.
Let $\varsigma=(W,\hQ)$ and we will use $\psi^\pm_\varsigma$ to denote the phase function associated to $\varsigma$.
In fact, $\hQ$ can be explicitly defined as follows:
$\hQ(\psi)$ consists of those objects $M$ satisfying

\begin{enumerate}
  \item $M$ is $W$-semistable in the quasi-abelian category $\hP(a,b)$ for
  some real numbers $a, b$ such that
  \[\begin{cases}
      0<b-a<1-2\epsilon, \\
      [\psi-\epsilon,\psi+\epsilon]\subset(a,b).
    \end{cases}
  \]
  \item  $W(M)=m\cdot e^{\bfi \pi \psi}$ for some $m\in\RR_+$.
\end{enumerate}

Now let us come back to the proof of our fusion-stable version.
Suppose that $\sigma=(Z,\hP)$ is $\Fu$-stable.
Take any $W\in\Hom_{\KG(\Fu)}(\KG(\D),\CC_{\Fu})$ satisfying
\begin{equation}\label{eq:W2}
    | W(M)-Z(M) |<\sin(\pi\epsilon/2)|Z(M)|,
\end{equation}
which uniquely determines a stability condition $\varsigma=(W,\hQ)$ with $d(\sigma,\varsigma)<\epsilon/2$.
Then we only need to show that $\hQ$ (and hence $\varsigma$) is also $\Fu$-stable,
which implies that $\Fu$-stable subspace of $\Hom_{\ZZ}(\KG(\D),\CC)$ gives the local coordinates of $\Stab^\Fu(\D)$.
We will use the same strategy of the proof of \cite[Prop.~3.3.4]{H2}.

Notice that $W$ induces a skewed stability function on the quasi-abelian category $\hP(a,b)$.
Let $\Sim\Fu=\{x_1,\ldots,x_m\}$ and $M$ be any $\varsigma$-semistable object with phase $\psi$.
By construction of $\hQ(\psi)$, $M$ is $W$-semistable in some quasi-abelian category $\hP(a,b)$ as above.
We only need to show that $M_i\colon=x_i(M)$ is also $\varsigma$-semistable for any $1\le i\le m$,
or equivalently, $M_i$ satisfies $1^\circ\sim2^\circ$ above.

Since $\sigma$ and $W$ are both $\Fu$-stable, then $M\in\hP(a+\epsilon,b-\epsilon)$
implies that $M_i\in\hP(a+\epsilon,b-\epsilon)$ and
$W(M_i)=m_i\cdot e^{\bfi \pi \psi}$ for some $m_i\in\RR_+$.
The only less obvious thing to prove is the $W$-semistability of $M_i$ (in $\hP(a,b)$).

Suppose that some $M_i$ is not $W$-semistable.
As $d(\sigma,\varsigma)<\epsilon/2$, $M_i\in\hQ(a+\epsilon/2,b-\epsilon/2)$.
Let $r=\max\{\psi^+_\varsigma(M_i)\mid 1\le i\le m\}$ and then $r\in (\psi,b-\epsilon/2)$.
Without loss of generality, suppose that $M_m$ admits a HN-filtration with first factor $A\in\hQ(r)$.
So $A$ is $W$-semistable with a strict monomorphism $f_m: A \to M_m$ in $\hQ(a+\epsilon/2,b-\epsilon/2)$.
Denote by $x_m^\vee$ the (left/right) dual of $x_m$ and we have
\[
    \Hom( x_m^\vee(A), M )\cong\Hom( A,x_m(M) )\neq0
\]
by the adjunction property.
As $W$ is $\Fu$-stable, $W(x_m^\vee(A))$ has phase $r=\psi_\varsigma(A)$,
which is bigger than $\psi=\psi_\varsigma(M)$.
Since $M_m\in\hQ(a+\epsilon/2,b-\epsilon/2)$,
we know that $A\in\hQ(a+\epsilon/2,b-\epsilon/2)$ and
hence $A\in\hP(a,b)$ as well as $x_m^\vee(A)\in\hP(a,b)$.
Now in $\hP(a,b)$, $M$ is $W$-semistable and thus $x_m^\vee(A)$ is not $W$-semistable.
So there exists a $W$-semistable object $B$ with phase $r'>r$ and a strict monomorphism
$f:B\to x_m^\vee(A)$ in $\hP(a,b)$.
We then obtain a strict monomorphism
\[
    x_m^\vee(f_m)\circ f\colon B\to x_m^\vee(A) \to x_m^\vee(M_m)=x_m^\vee( x_m(M) ).
\]
Thus, $\psi^+_\varsigma( x_m^\vee( x_m(M) ) )\ge r'$.
However on the other hand, one has
\[
    x_m^\vee( x_m(M)=\bigoplus_{i=1}^m M_i^{\oplus k_i}
\]
for some $k_i\in\NN$, which implies that $\psi^+_\varsigma( x_m^\vee( x_m(M) ) )$ is at most $r$ (by definition of $r$).
There is the contradiction, which finishes the proof.
\end{proof}

To finish this section, we also give a generalization of \cite[Thm.~A]{QW} for the later use.

\begin{theorem}\label{thm:QW+}
Let $\Stab_0$ be a finite type component as in \eqref{eq:0}.
Then $\Stab_0^\Fu$ consists of finite type components, each of which is contractible.
\end{theorem}
\begin{proof}
The key of the contractibility in \cite[Thm.~A]{QW} is the CW-ish complex structure
(cf. the stratification in \cite[\S~3.1]{QW}).
Namely:
\begin{itemize}
  \item The top cells in $\Stab_0$ are (half open half closed) $U(\h)\cong\mathbb{H}^{\Sim\h}$ in \eqref{eq:0}.
  \item The closure of two top cells intersect at a $k$-codimension wall if and only if there is a tilting
  $\mu^\sharp_\hF\colon\h\to\h^\sharp$ between the corresponding hearts, where
  $\hF$ is generated by $k$ simples in $\h$.
\end{itemize}
Now, such a structure is inherited by $\Stab_0^\Fu$, where for any $\Fu$-stable heart $\h$,
$\Sim\h$ groups into $d=\dim\Hom_{\KG(\Fu)}(\KG(\D),\CC_{\Fu})$ sets,
each of which corresponds to a $\Fu$-simple (subcategory) in $\h$.
Then one has
\begin{itemize}
  \item The top cells in $\Stab_0^\Fu$ are (half open half closed) $U(\h)^\Fu\cong\mathbb{H}^d$.
  \item The closure of two top cells intersect at a $k$-codimension wall if and only if there is a tilting
  $\mu^\sharp_\hF\colon\h\to\h^\sharp$ between the corresponding hearts, where
  $\hF$ is generated by $k$ $\Fu$-simples in $\h$.
\end{itemize}
The other finiteness condition, that there are only finitely many torsion pairs in any $\h\in\EG_0$,
is also inherited. 
Then applying the arguments in \cite{QW},
we extend contractibility result there to the $\Fu$-stable version.
More precisely, here is the outline of the proof:
\begin{itemize}
  \item There is a stratification of $\Stab_0^\Fu$ (cf. \cite[(4) in \S~3.1]{QW}),
  such that the walls between cells are precisely encoded by $\Fu$-stable tilting of hearts.
  \item The stratification is locally finite and closure-finite (\cite[Lem~4.4]{QW}).
  \item Thus $\Stab$ is a regular, totally normal CW-cellular stratified space (\cite[Prop~3.21]{QW}).
  \item So the proof of \cite[Thm~4.9]{QW} applies and $\Stab_0^\Fu$ is contractible.
\end{itemize}
\end{proof}

\begin{remark}
As shown in \Cref{thm:KQ}, \Cref{thm:stab} and \Cref{thm:QW+},
one should be able to generalize most statements about tilting/cluster/stability structures to a fusion-stable version, similar to the group-stable case.
\end{remark}

\section{Applications}
Throughout this section,
let $\spec$ be a finite-type weighted quiver, i.e. an oriented Coxeter graph listed in \Cref{fig:Cox}.
Suppose that we have a weighted folding \eqref{eq:f:Q2Q}
as one of cases in \Cref{ex:w-folding}.

\subsection{Categorification of cluster exchange graphs}\

Denote by $\uCEG(\spec)$ cluster exchange graph of $\spec$,
also known as the generalized associahedron in the sense of \cite{FR}.

\begin{example}
The cluster exchange graphs $\uCEG(\spec)$ of types $A_3, B_3/C_3$ and $H_3$ are shown in \Cref{fig:ABH}.
\end{example}

We show that the categorical construction of $\uCEG(\spec)$ given in \cite{DT1,DT2}
is naturally the fusion-stable cluster exchange graph,
which explain the relation between the works \cite{DT1,DT2} and \cite{H1}, as asked in \cite[\S~6]{H1}.

\begin{theorem}\label{thm:ceg}
Let $f$ be a finite weighted folding, listed in \Cref{ex:folding} or \Cref{ex:w-folding}.
Then there is a canonical identification between cluster exchange graphs
\[
    \uCEG(\spec)\cong\uCEG_2(Q)^\Fu
\]
\end{theorem}
\begin{proof}
By \cite{DT1,DT2}, vertices and edges of $\uCEG(\spec)$ can be realized as certain vertices/paths of
$\uCEG(Q')$ for some unfolded $Q'$.
More precisely, the edges/mutations in $\uCEG(\spec)$ correspond to simultaneous mutations in $\uCEG(Q')$.
Notice that the unfolded quiver $Q'$ is the same as ours (or some connected component of ours).

When the orientation of $Q\supset Q'$ is compatible with weighted folding,
the initial cluster tilting object $\mathbf{Y}_Q$ is $\Fu$-stable.
Furthermore, \Cref{lem:G-Ind} implies that $\Fu$-ind forward mutation is simultaneously mutating
many summands of a $\Fu$-stable cluster tilting object $\mathbf{Y}$ (which are in the same $\Fu$-ind orbit).
This coincides with the multi-mutation in \cite{DT1,DT2}.
Hence, the theorem follows.
\end{proof}


A closely related work is \cite{HHQ},
which shows $\on{CT}(\spec)\cong\Br_{\spec}$,
where $\on{CT}(\spec)$ is the cluster braid group, defined as the fundamental group of a groupoid version of $\CEG_2(Q)^\Fu$ by gluing many $(m+2)$-gons for each $m$ appearing as weights in $\spec$.
See there for more details.
In fact, $\Br_{\spec}$ will also appear in the next application.

\subsection{Hyperplane arrangements of finite Coxeter-Dynkin type}\
\def\fkh{\mathfrak{h}}
\def\fkhd{\mathfrak{h}_\Delta}

\def\reg{{\on{reg}}}
\def\hreg{\fkh^\reg}
\def\UC{\widetilde{\fkh^\reg_\Delta/W_\Delta}}
\def\dq{\Delta_{\spec}}

\paragraph{\textbf{Root systems and hyperplanes}}\

Let $\Delta$ be a Coxeter diagram (with edge weight $\wt$).
Denote by $\Lambda_\Delta$ the root system associated to $\Delta$, living in $\RR^\bfn$.
Denote by $\fkhd\cong\CC^\bfn$ the complexification of dual space of $\RR^\bfn$.
So $\Lambda_\Delta\subset\fkhd^*$ and each root $\alpha\in\Lambda$ determines a hyperplane
\[
    H_\alpha=\{\varsigma\in\fkhd\mid \varsigma(\alpha)= 0\}.
\]
Let $\hreg_\Delta\subset\fkhd$ be the complement of the root hyperplanes, i.e.
\begin{equation}\label{eq:hreg}
    \hreg_\Delta\colon= \fkhd\setminus
    \big(\bigcup_{\alpha\in\Lambda_\Delta}H_\alpha\big).
\end{equation}

Recall we have defined the braid and Weyl group associated to $\Delta$.

It is well-known that $W_\Delta$ can be geometrically realized as the reflection group on $\hreg_\Delta$,
where the generators $r_{s_{\bfi}}$ are the reflections with respect to simple roots of $\Delta$. Moreover, $W_\Delta$ acts freely on $\hreg_\Delta$ and one has
\begin{gather}\label{eq:pi1=Br}
    \pi_1(\hreg_\Delta/W_\Delta)=\Br_\Delta.
\end{gather}

A fundamental domain for $\hreg_\Delta/W_\Delta$ can be chosen to be the Weyl chamber
\[
    U_\Omega=\{ \theta\in\fkhd  \mid
        \on{Im} \alpha(\theta)\ge0,\;\forall  \alpha\in\Lambda_\Delta \}.
\]

For simply laced types ADE,
\cite{B2} shows that the space of stability conditions on 2-Calabi-Yau category associated to $\Delta$
is a covering space of $\hreg_\Delta/W_\Delta$
and \cite{BT} shows that it is in fact the universal cover.
We shall generalize this to non-simply laced types.

\paragraph{\textbf{Weighted folding root systems}}\

Consider a weighted folding
\begin{gather}\label{eq:ul f}
    \wtF\colon \Omega\to\Delta
\end{gather}
from a (possibly a disjoint union of) ADE diagram $\Omega$ to a non-simply laced diagram $\Delta$.
So we have that
\[
    \Omega=\underline{Q} \quad\text{and}\quad \Delta=\underline{\spec}
\]
are the underlying graphs of the quivers $Q$ and $\spec$, respectively.
We restricted to the cases in \Cref{ex:folding} or \Cref{ex:w-folding}.
For $\Omega$, we have the corresponding root system $\Lambda_\Omega$ and associated $\fkh_\Omega$, $W_\Omega$, etc.

\begin{example}\label{ex:H34}
Here are some examples of showing (the projection of) weighted folding of roots systems.
\begin{itemize}

\item The root system of $D_6$ decomposes into proportional (with golden ratio) copies of type $H_3$,
as shown in \Cref{fig:D6=2H3}--one in blue and one in orange.
The middle picture is the $D_6$ root system (which is putting the blue and orange ones together)
while each black vertex is the projection of two roots.

\begin{figure}[ht]\centering
\includegraphics[width=\textwidth]{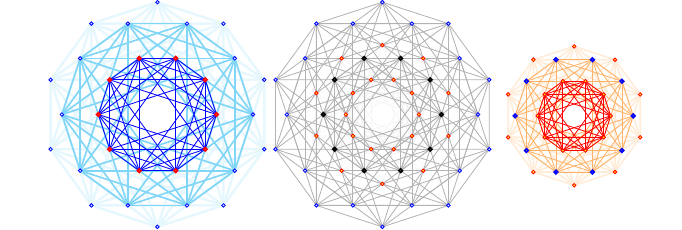}
\caption{The root systems of
    $\Lambda_{D_6}=\frac{\sqrt{5}+1}{2}\cdot \Lambda_{H_3}\bigsqcup\Lambda_{H_3}$}
\label{fig:D6=2H3}
\end{figure}
\begin{figure}[ht]\centering
  \centering\makebox[\textwidth][c]{
    \includegraphics[width=18cm]{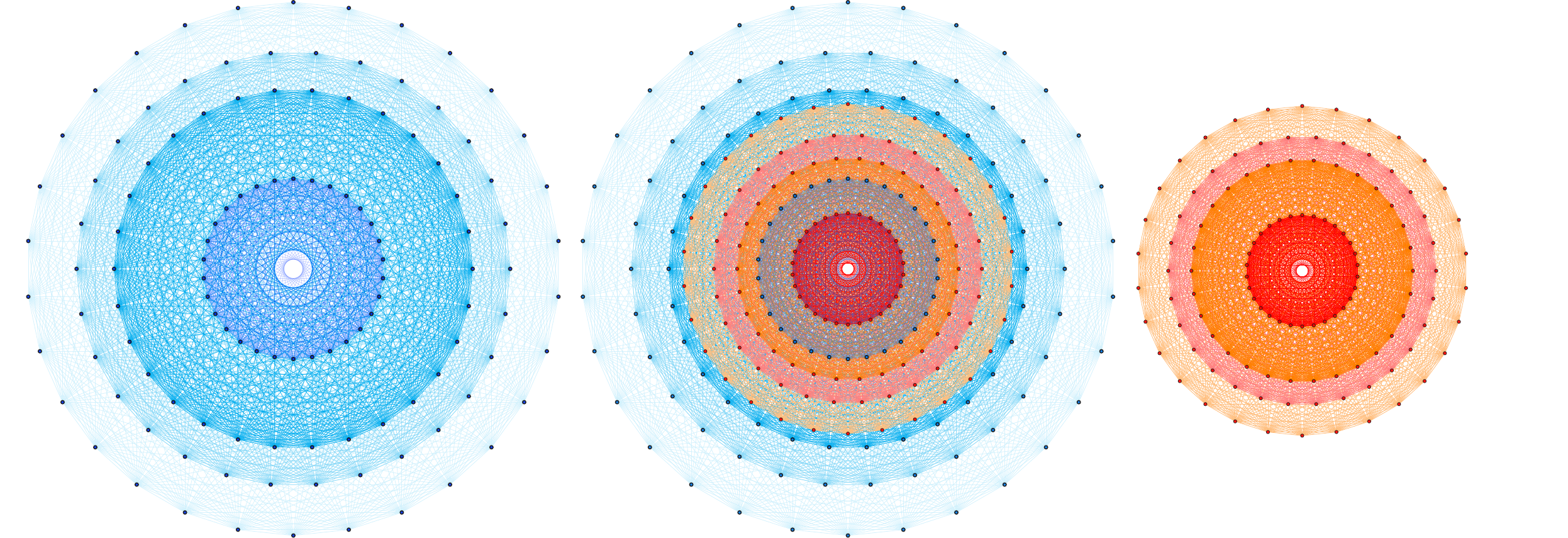}
  }
\caption{The root systems of
    $\Lambda_{E_8}=\frac{\sqrt{5}+1}{2}\cdot \Lambda_{H_4}\bigsqcup\Lambda_{H_4}$}
\label{fig:E8=2H4}
\end{figure}

\item the 120 roots of type $H_4$ form the vertex set of a 600 cell.
And the root systems of type $E_8$ consists of two proportional (with golden ratio) copies of
    roots systems of type $H_4$, as shown in \Cref{fig:E8=2H4}.
\end{itemize}
\end{example}

We can identify $\fkhd^*=(\fkh_\Omega^*)^\Fu$.
Explicitly, each vertex $\bfi\in\Delta_0$ corresponds to a simple root $\bfs_{\bfi}\in\Lambda_\Delta$ and
$\wtF^{-1}(\bfi)$ is a collection of simple roots of $\Omega$.
One can choose $\bfs_{\bfi}$ in each $\Fu$-stable subspace
\begin{gather}\label{eq:s}
    \RR\cdot\sum_{s\in \wtF^{-1}(\bfi)} s
\end{gather}
so that all $\bfs_{\bfi}$ form the simple roots of a root system of type $\Delta$,
which we will identify it with $\Lambda_\Delta$.
For instance, in type $H$, one can just take $\bfs_{\bfi}=\sum_{s\in \wtF^{-1}(\bfi)} s$,
cf. \Cref{ex:H34}.
Thus we have $\fkhd^*=(\fkh_\Omega^*)^\Fu$.

Then
\[
    \fkh_{\Delta}=\Hom(\fkhd^*,\CC)=\Hom(\fkh_\Omega^*,\CC)^\Fu
        =\fkh_\Omega^\Fu
\]
and hence (restricted to the fundamental domain) we have
\begin{equation}\label{eq:U=U}
    U_\Delta=U_\Omega^\Fu.
\end{equation}

\paragraph{\textbf{Weighted folding braid/Weyl groups}}\

On the level of braid/Weyl groups, there are also induced injections
(cf. \eqref{eq:s}).
\begin{lemma}\cite{Cr}\label{thm:Cr}
The weighted folding \eqref{eq:ul f} induces an injection
\begin{gather}\label{eq:iota}
\begin{array}{rcl}
  \wtF^*\colon \Br_\Delta & \hookrightarrow & \Br_\Omega.\\
                            b_\bfi & \mapsto & \prod_{i\in \wtF^{-1}(\bfi) } b_i
\end{array}
\end{gather}
\end{lemma}

Similarly, we also have $\wtF^*\colon W_\Delta \hookrightarrow W_\Omega$.
In particular,
when identifying $W_\Omega$ with reflection group on $\hreg_\Delta$,
we have
\[
    r_{\bfs_{\bfi}}=\prod_{i\in \wtF^{-1}(\bfi) } r_{s_i}.
\]

\paragraph{\textbf{2-Calabi-Yau categories}}\

Let $\D_2(\Omega)=\D_2(Q)$ be the 2-Calabi-Yau category associated to $Q$ (cf. \Cref{sec:quivery}).
When $Q$ is connected, $\D_2(Q)$ can be alternatively constructed as a subcategory of
the 2-Calabi-Yau category associated to the corresponding Kleinian singularity \cite{B2}.
However, we will allow the case when $Q$ is a union of ADE quivers, which does not really make a difference.
$\D_2(\Omega)$ admits a canonical heart $\h_\Omega$ with simples $\{S_i \mid i\in \Omega_0\}$,
which are $2$-spherical objects.
Recall the spherical twist functor associated to any spherical object $S$ is defined as
\begin{equation}\label{eq:ST}
    \Phi_S(X)=\Cone\left(S\otimes\Hom^\bullet(S,X)\to X\right)
\end{equation}
with inverse
\[
    \Phi_S^{-1}(X)=\Cone\left(X\to S\otimes\Hom^\bullet(X,S)^\vee \right)[-1].
\]
The spherical twists of these simples generate the spherical twist group $\ST(\Omega)$.
By \cite[Thm.~3.1]{BT} (cf. \cite[Thm.~B]{QW}), there is a natural isomorphism:
\begin{gather}\label{eq:iota O}
\begin{array}{rcl}
  \iota_\Omega\colon \Br_\Omega & \cong & \ST(\Omega)\\
                            b_i & \mapsto & \Phi_{S_i}.
\end{array}
\end{gather}

The Grothendieck group $\KG(\D_2(\Omega))$ equipped with the Euler form $\<-,-\>_\Omega$ can be identified with the root lattice $\Lambda_Q$ equipped with the Killing form $\chi_\Omega(-,-)$,
where the simple roots correspond to simple objects (cf. \cite[\S~1.2]{B2}):
\begin{gather}\label{eq:iota2}
\begin{array}{rcl}
  \iota_Q\colon \Lambda_\Omega & \cong & \KG(\D_2(\Omega))\\
                            s_i & \mapsto & [S_i].
\end{array}
\end{gather}
As the two forms match,
the spherical twist \eqref{eq:ST} of $S_i$ acts on $\KG(\D_2(\Omega))$ as
reflection with respect to $s_i$ on $\Lambda_\Omega$,
in the sense that the following diagram commutes:
\begin{gather}\label{eq:fig}
\begin{tikzpicture}[xscale=.5,yscale=.5]
\draw(180:3)node(o)
    {$\Lambda_\Omega$}   (-3,2.5)node(b)
    {\small{$W_\Omega$}}   (0:3)node(a)
    {$\KG(\D_2(\Omega))$}   (3,2.5)node(s)
    {\small{$\ST(\Omega)$}};
\draw[-stealth](o)to node[above]{$\cong$}(a);
\draw[stealth-](b)to node[above,white]{$\pi_\Omega$}(s);
\draw[-stealth](-3.2,.6).. controls +(135:2) and +(45:2) ..(-3+.2,.6);
\draw[-stealth](3-.2,.6).. controls +(135:2) and +(45:2) ..(3+.2,.6);
\end{tikzpicture}.
\end{gather}

As mentioned above, for simply laced ADE diagram $\Omega$,
there is an isomorphism (\cite[Thm.~1.2]{B2})
\begin{gather}\label{eq:ciso}
    c_\Omega\colon \Stap\D_2(\Omega)/\ST(\Omega)\to\hreg_\Omega/W_\Omega.
\end{gather}
A fundamental domain for $\Stap\D_2(\Omega)/\ST_\Omega$ is $U(\h_\Omega)$,
consisting of $\sigma$ whose heart is $\h_\Omega$.
As shown in \cite{B2}, the key to establish \eqref{eq:ciso} is to identify
\begin{gather}\label{eq:corr}\begin{array}{rcl}
   U_\Omega^\circ & \xrightarrow{\cong} & U(\h_\Omega)^\circ \\
        \theta & \mapsto & \sigma_\theta=(Z_\theta,\h_\Omega),
\end{array}\end{gather}
such that
\begin{gather}\label{eq:Z=theta}
    Z_\theta ( [S_i] ) = \theta( s_i ),
\end{gather}
where $?^\circ$ denotes the interior of $?$.
Further, \eqref{eq:iota O} implies that $\Stap\D_2(\Omega)$ is in fact a universal cover
(noticing \eqref{eq:pi1=Br}).
We end this section by generalizing such a result to the non-simply laced diagram $\Delta$ via fusion-stable stability structure.

\paragraph{\textbf{Main application}}\

Recall from \Cref{sec:fun.action},
there is an action $\Fu$ on $\D_2(\Omega)$.
We proceed to show the main result of this section.

\begin{theorem}\label{thm:Kleinian}
Let $\Delta$ be a finite Coxeter-Dynkin diagram
with an associated fusion weighted folding \eqref{eq:f:Q2Q}, listed as in \Cref{ex:folding} or \Cref{ex:w-folding}.
Then there is a natural isomorphism $c_\Delta$ between complex manifolds
\begin{equation}\label{eq:thm}
\begin{tikzpicture}[xscale=.7,yscale=.7]
\draw(180:3)node(o)
    {$\widetilde{\hreg_\Delta/W_\Delta}$}   (-3,2.2)node[blue](b)
    {\small{$\Br_\Delta$}}   (0:3)node(a)
    {$\Stap(\D_2(\Omega))^\Fu$}   (3,2.2)node[blue](s)
    {\small{$\ST(\Delta)$}};
\draw[-stealth](o)to node[above]{$c_\Delta$} node[below]{$\cong$}(a);
\draw[-stealth,blue](b)to node[above]{\small{$\cong$}}(s);
\draw[-stealth,blue](-3.2,.6).. controls +(135:2) and +(45:2) ..(-3+.2,.6);
\draw[-stealth,blue](3-.2,.6).. controls +(135:2) and +(45:2) ..(3+.2,.6);
\end{tikzpicture}
\end{equation}
with identical braid/spherical twist action.
Moreover, $\Stab(\D_2(\Omega))^\Fu$ is the universal cover of $\hreg_\Delta/W_\Delta$, which is contractible.
\end{theorem}
\begin{proof}
For unfolded (ADE type) $\Omega$, the theorem holds by \eqref{eq:ciso}, faithfulness of spherical twist group action in \cite{BT} and \cite[Thm.~A]{QW}.
Now we need to take the $\Fu$-stable version of all the things discussed above.
Namely:
\begin{itemize}
  \item For the spherical twist, combining \eqref{eq:iota O} with \Cref{thm:Cr}, one obtain the following commutative diagram:
\begin{equation}\label{eq:cpt}
\begin{tikzcd}
  \Br_\Delta \ar[r,"\cong"',"\iota_\Delta"] \ar[d,hookrightarrow,"\wtF^*"'] & \ST(\Delta) \ar[d,hookrightarrow,""]\\
  \Br_\Omega \ar[r,"\cong"',"\iota_\Omega"] & \ST(\Omega),
\end{tikzcd}
\end{equation}
where $\ST(\Delta)$ is the image of $\iota_\Omega\circ \wtF^*$
and $\iota_\Delta$ sends $b_\bfi$ to $\prod_{i\in \wtF^{-1}(\bfi) } \Phi_{S_i}$.
  \item For the fundamental domains, noticing \eqref{eq:U=U}, one restricts \eqref{eq:corr} to the $\Fu$-stable part that gives an isomorphism $U_\Delta^\circ  \xrightarrow{\cong}  U(\h_\Delta)^\circ$,
      where $\h_\Delta=\h_\Omega$ is the initial heart (which is $\Fu$-stable).
\end{itemize}
What is left to check is
\begin{itemize}
  \item[($\star$)] $\ST(\Delta)$ indeed acts on $\Stab^\Fu(\D_2(\Omega))$, i.e.
  preserves the $\Fu$-stable property.
\end{itemize}
Then the compatibility \eqref{eq:cpt} between spherical twist group and Weyl group for $\Delta$
implies the required isomorphism.
Finally, the contractibility follows from the finiteness property, i.e. Theorem~\ref{thm:QW+}.

\textbf{For ($\star$):}
We first notice the fact that spherical twist actions on hearts are equivalent to
(a sequence of) simple tilting for $\D_2(\Omega)$ (cf. \cite{B2}).
By \eqref{eq:iota} and \eqref{eq:cpt}, we see that
the inverse of any initial generator $\phi^{-1}_{\bfi}$ of $\ST(\Delta)$ corresponds to
a product of spherical twists $\phi_i=\phi_{S_i}$ for $1\le i\le l$,
where $\wtF^{-1}(\bfi)=\{1,\ldots,l\}$ and $\{S_i\}_{i=1}^l$ are in the same $\Fu$-simple subcategory
(denoted by $\Fu(S_{\bfi})$).
Using formula \eqref{eq:G-simple} in \Cref{lem:G-simple}, we have
\[
    \phi^{-1}_{\bfi}(\h_\Delta)=\big( \prod_{i=1}^l \phi_{i}^{-1} \big) (\h_\Delta)
    =\big( \prod_{i=1}^l \mu^\sharp_{S_i} \big) (\h_\Delta)
    =\mu^\sharp_{ \Fu(S_{\bfi}) } (\h_\Delta).
\]
Thus $\phi^{-1}_{\bfi}(\h_\Delta)$ is also $\Fu$-stable.
Inductively as in \cite[Cor.~8.5]{KQ1},
the products of spherical twists become products of simple tiltings and
one deduces that at each step, the product of simple tiltings is in fact a $\Fu$-stable one.
Hence, $\ST(\Delta)$ preserves the $\Fu$-stable property.
\end{proof}

As a corollary, we extend the alternative proof of the following theorem, cf. \cite[\S~5.2]{QW}, to the non-simply laced case.
\begin{theorem}\cite{D}
The universal cover of $\hreg_\Delta/W_\Delta$ is contractible.
\end{theorem}


\end{document}